\documentclass[11pt,reqno, oneside]{amsart}
\usepackage{etex}
\usepackage[T1]{fontenc}
\usepackage{lmodern}
\usepackage[left=1in, right=1in,top=1in,bottom=1in, marginparwidth=3cm]{geometry}
\usepackage{stackengine}
\usepackage{stmaryrd}
\usepackage{color}
\usepackage[latin1]{inputenc}
\usepackage[T1]{fontenc}
\usepackage[normalem]{ulem}
\usepackage[english]{babel}
\usepackage{verbatim}
\usepackage{graphicx}
\usepackage{enumerate}
\usepackage{amsmath,amssymb,amsfonts,amsthm,amscd,mathrsfs}
\usepackage{mathtools}

\usepackage{thmtools}
\usepackage{thm-restate}
\usepackage{array}
\usepackage{bbm,bm}
\usepackage[mathscr]{euscript}
\usepackage{microtype}
\usepackage{dsfont}
\usepackage[T1]{fontenc}
\usepackage{babel}
\usepackage{url}
\usepackage{caption}
\usepackage{subcaption}
\usepackage{mathbbol}
\DeclareSymbolFontAlphabet{\mathbb}{AMSb}%
\usepackage[bookmarksopen, bookmarksnumbered]{hyperref}
\usepackage[dvipsnames]{xcolor}
\usepackage{tikz,ifthen}
\usetikzlibrary{intersections}
\usetikzlibrary{calc}
\usetikzlibrary{decorations.markings}
\usetikzlibrary{decorations.pathreplacing}
\usetikzlibrary{arrows}
\usetikzlibrary{snakes}
\usetikzlibrary{overlay-beamer-styles}
\usetikzlibrary{external}
\usetikzlibrary{matrix}
\usetikzlibrary{shapes.misc,calc, arrows.meta, fit, hobby}
\usetikzlibrary{decorations.pathmorphing}

\usetikzlibrary{decorations.markings,fpu}
\makeatletter
\tikzset{use fpu reciprocal/.code={%
\def\pgfmathreciprocal@##1{%
    \begingroup
    \pgfkeys{/pgf/fpu=true,/pgf/fpu/output format=fixed}%
    \pgfmathparse{1/##1}%
    \pgfmath@smuggleone\pgfmathresult
    \endgroup
}}}%
\makeatother

\usepackage{cleveref}

\newtheorem{theorem}{Theorem}[section]

\newtheorem{lemma}[theorem]{Lemma}

\newtheorem{proposition}[theorem]{Proposition}
\newtheorem{corollary}[theorem]{Corollary}

\declaretheorem[name=Theorem, sibling=theorem]{theorem-restate}

\theoremstyle{definition}

\newtheorem{remark}[theorem]{Remark}

\input xy
\xyoption{all}

\allowdisplaybreaks

\DeclareMathOperator*{\emaxop}{max}
\NewDocumentCommand{\emax}{oe{_}}{%
  \IfNoValueTF{#2}
    {\emaxop\nolimits\IfValueT{#1}{^{#1}}}
    {\IfNoValueTF{#1}{\emaxop_{#2}}{\emaxcomplex{#1}{#2}}}%
}
\makeatletter
\newcommand{\emaxcomplex}[2]{\mathop{\mathpalette\emax@{{#1}{#2}}}}
\newcommand{\emax@}[2]{\emax@@{#1}#2}
\newcommand{\emax@@}[3]{% #1 = style, #2 = superscript, #2 = subscript
  \ifx#1\displaystyle\emax@@@{#2}{#3}\else\emaxop_{#3}^{#2}\fi
}
\newcommand{\emax@@@}[2]{%
  \begingroup\m@th
  \sbox0{$\displaystyle\emaxop$}%
  \sbox2{$\displaystyle\emaxop_{#2}$}%
  \dimen@=\dimexpr(\wd2-\wd0)/2\relax
  \sbox4{$^{#1}$}%
  \ifdim\wd4>\dimen@ \dimen@=\dimexpr\wd4-\dimen@ \else \dimen@=0pt\fi
  \operatorname*{max^{#1\kern-\wd4}}_{#2}\kern\dimen@
  \endgroup
}
\makeatother

\newcommand{\mrm}{\mathrm}
\newcommand{\msf}{\mathsf}
\newcommand{\mc}{\mathcal}
\newcommand{\mb}{\mathbf}

\newcommand{\E}{\mathbb E}
\newcommand{\PP}{\mathbb P}
\renewcommand{\P}{\mathbb P}
\newcommand{\ZZ}{\mathbb Z}
\newcommand{\Z}{\mathbb Z}
\newcommand{\RR}{\mathbb R}
\newcommand{\R}{\mathbb R}

\newcommand{\NN}{\mathbb N}
\newcommand{\N}{\mathbb N}

\newcommand{\rDe}{\mathring{\Lambda}}

\newcommand{\don}{\mathds{1}}

\newcommand{\cD}{\mathcal D}
\newcommand{\cK}{\mathcal K}
\newcommand{\cP}{\mathcal P}

\newcommand{\cL}{\mathcal L}
\renewcommand{\L}{\mathcal L}
\newcommand{\cM}{\mathcal M}

\newcommand{\cZ}{\mathcal Z}

\newcommand{\F}{\mathcal F}
\newcommand{\Fext}{\mathcal F_\mathrm{ext}}

\newcommand{\one}{\mathbbm{1}}

\newcommand{\bx}{\mathbf x}
\newcommand{\by}{\mathbf y}

\newcommand{\fh}{\mathfrak h}

\newcommand{\midd}{\ \Big|\ }
\newcommand{\middd}{\ \bigg|\ }
\newcommand{\intint}[1]{\llbracket#1\rrbracket}
\newcommand{\diff}{\mathrm{d}}
\newcommand{\h}{\mathfrak h}
\newcommand{\hh}{\hat{\mathfrak h}}

\newcommand{\qq}[1]{[\![{#1}]\!]}



\numberwithin{equation}{section}

\title{
van den Berg-Kesten--type correlation inequalities for disjoint polymers in the KPZ universality class}
\author{Shirshendu Ganguly}%
\thanks{Department of Statistics, UC Berkeley, Berkeley, CA, USA. e-mail: sganguly@berkeley.edu}
\author{Milind Hegde}%
\thanks{Division of Mathematical Sciences, School of Physical and Mathematical Sciences, Nanyang Technological Uni-
versity, Singapore. e-mail: milind.hegde@ntu.edu.sg}
\author{Lingfu Zhang}%
\thanks{The Division of Physics, Mathematics and Astronomy, California Institute of Technology, Pasadena, CA, USA. e-mail: lingfuz@caltech.edu}
\date{}

\begin{document}

\begin{abstract}
In classical percolation theory, the van den Berg-Kesten (BK) inequality is a fundamental tool that shows that disjoint events induce negative conditionings on each other. 
The inequality also holds in the context of  last passage percolation (LPP), which is the zero temperature limit of polymer models and an important subclass in the Kardar-Parisi-Zhang (KPZ) universality class. Recently, an analog of the BK inequality was discovered in the context of zero temperature line ensembles and the scaling limit of LPP, where it was used to study upper tail probabilities of the weight and the scaling limit of geodesics under such upper tail conditionings. 
However, while it has become apparent that such an inequality in the positive temperature setting would have a number of applications, it seems likely that a direct generalization of the zero temperature inequality would not hold.
In this work we prove a version of the BK inequality for the KPZ line ensemble and the continuum directed random polymer. We do so by working with the log gamma polymer, making use of its integrability and the geometric RSK correspondence. 
Our inequality serves as a key input in analyzing the KPZ line ensemble and proving sharp upper tail estimates of the KPZ equation in \cite{GH22}, and proving convergence of the continuum directed random polymer to Brownian bridge under the upper tail event in \cite{ganguly2023brownian}. The crucial role of integrability in the validity of such an inequality is highlighted via a counter-example for a non-integrable model.
\end{abstract}

\maketitle

\setcounter{tocdepth}{1}

\tableofcontents{}

\section{Introduction and main results}\label{s.intro}

\subsection{Background on the zero temperature BK inequality} The van den Berg-Kesten (BK) \cite{van1985inequalities,reimer2000proof} inequality is a well-known and important tool in classical Bernoulli percolation on $\Z^d$. It says that the probability of two events occurring ``disjointly'' (i.e., their occurrence can be verified by inspecting the values of the i.i.d.\ random variables associated to disjoint subsets of the graph) is upper bounded by the product of the individual events' probabilities, i.e., they are negatively correlated. 

While in its original form it was stated in the case that the random variables are Bernoulli-distributed, it also extends to the case of general distributions on $\R$ \cite{arratia2018van}. This has made it a useful tool also in the study of last passage percolation (in which one considers the maximum energies of directed paths passing through an i.i.d. random field, where the energy of a path is the sum of the variables it passes through) on $\Z^2$ \cite{basu2022interlacing,basu2022nonexistence,basu2023exponent,schmid2023mixing,alberts2025large}, an important subclass of models in the Kardar-Parisi-Zhang universality (KPZ) class.

Another important tool in the study of the KPZ universality class are \emph{line ensembles}, or collections of random interacting functions admitting a local Gibbs property. Two central examples are the parabolic Airy line ensemble \cite{CH11} and the KPZ line ensemble \cite{CH14}. The parabolic Airy line ensemble is a zero temperature object associated to a limiting LPP model known as the directed landscape, while the KPZ line ensemble is a positive temperature object associated to a limiting polymer model known as the continuum directed random polymer (CDRP). Both line ensembles have curves indexed by $\N$, where the $k$ lowest indexed curves together encode information about the energy of $k$ \emph{disjoint} paths for every $k$.

This connection to disjoint paths suggests that the above line ensembles should enjoy a version of the BK inequality. This was proven for the case of the parabolic Airy line ensemble \cite{GH22} (though its proof operates at the level of the line ensemble directly). The setting of the KPZ line ensemble is significantly more subtle, as will become apparent shortly, and proving a version of the BK inequality in that setting is the goal of this paper; the question of the validity of the BK inequality for the KPZ equation was also raised by Xuan Wu to the authors at the KPZ meets KPZ workshop at the Fields Institute in 2024. 

Let us briefly state the inequality proved for the parabolic Airy line ensemble $(\cP_1, \cP_2, \ldots)$ in \cite{GH22}. Let $a<b$, $\mc C([a,b],\R)$ be the space of real-valued continuous functions on $[a,b]$, and $\msf A$ be an increasing Borel measurable subset of $\mc C([a,b],\R)$, where increasing means that if $f\in \msf A$ and $g\in\mc C([a,b],\R)$ satisfies $g\geq f$ pointwise, then $g\in \msf A$. The version of the BK inequality from \cite{GH22} says that, almost surely,
\begin{align}\label{e.airy bk}
\P\left(\cP_2(\bm\cdot)|_{[a,b]} \in \msf A \midd \cP_1\right) \leq \P\left(\cP_1(\bm\cdot)|_{[a,b]} \in \msf  A\right);
\end{align}
in words, conditioned on the first curve, the second curve is stochastically dominated by an unconditioned copy of the first curve. This inequality was used in \cite{GH22} to derive sharp estimates on upper tail probabilities of the first curve (in fact, this argument was given for a general class of line ensembles satisfying certain assumptions, including a much weakened form of \eqref{e.airy bk}).

\begin{remark}\label{r.different BKs}
  We point out that, while the classical BK inequality suggests an inequality of the type of \eqref{e.airy bk} for the parabolic Airy line ensemble, it does not imply it. This is due to the fact that generic events of $\cP_2$ and $\cP_1$  encode energies of collections of two disjoint paths as a whole, and not of each path individually. However, certain special events can be related to events of the energies of the individual paths, for which, indeed, \eqref{e.airy bk} is implied by the classical BK inequality.%
\end{remark}

\subsection{The BK inequality in positive temperature}\label{s.bk in positive temp}
Given its great utility in LPP and other zero temperature contexts, it is of interest to obtain a generalization of the BK inequality in positive temperature settings, such as polymer models and their associated line ensembles. 
The curves of the positive temperature analog of $\cP$, the KPZ line ensemble, which we will denote $\h_{t} = (\h_{t,1}, \h_{t,2}, \ldots)$, are associated to the free energy of polymers or collections of disjoint polymers in a continuum polymer model. Here, informally, the free energy is the logarithm of the partition function, which is the integral over all directed paths of the exponential of the path's energy. Note the contrast with LPP, where one considers the single path with the maximum energy. Our goal is to obtain an inequality like \eqref{e.airy bk} with $\h_t$ in place of $\cP$.

The distinction just noted already poses a conceptual obstacle to obtaining a form of the BK inequality in the positive temperature context. Indeed, the weights of paths in LPP models are determined by just the portion of the environment they pass through, which works well with providing disjoint ``certificates'' of events occurring as needed in the classical BK inequality. The free energy in a polymer model, however, necessarily involves all the randomness in the environment, and so, for instance, while the largeness of $\cP_2$ can be tied to the existence of two disjoint paths of large energies, such a simple disjoint certificate is not available for the largeness of $\h_{t,2}$.
Thus no naive or immediate generalization of the BK inequality from LPP is possible. This issue can be understood as fundamentally arising from entropy.

Nevertheless, in this work we prove a form of the BK inequality for the free energy of the CDRP and the KPZ line ensemble. In the next section we introduce our main objects more precisely and in Section~\ref{s.results} we state our main results.

\begin{remark}
As we will see, the inequality we prove for the KPZ line ensemble differs slightly from \eqref{e.airy bk}, due to the complications arising from the entropy phenomenon in positive temperature. We will expand more on this point after the statements of the results are given.
\end{remark}

\subsection{Continuum directed random polymer \& KPZ line ensemble}  \label{ssec:CDRP}

In this section we define the CDRP's disjoint polymer partition functions and the KPZ line ensemble. We will not need the polymer measure for the CDRP.
For all of these objects, we take the inverse temperature $\beta=1$ throughout this paper for simplicity of notation, while our arguments go through verbatim for any fixed $\beta$.

\subsubsection{The CDRP partition function, multiplicative stochastic heat equation, and its multi-line extension}  \label{sss:pf}
Let $\xi$ be a space-time white noise, and define $\R^4_{\uparrow} := \{(x,s;y,t) \in \R^4 : s<t\}$.
For $(x,s;y,t) \in \R^4_{\uparrow}$, let $(x,s; y,t) \mapsto \cZ(x,s;y,t)$ be the (mild) solution to the multiplicative stochastic heat equation (SHE) defined by requiring, for all $x,s\in\R$,
\begin{equation}\label{e.SHE definition}
\begin{cases}
\partial_t \cZ(x,s;y,t) = \frac{1}{4}\partial_{y}^2 \cZ(x,s;y,t) + \cZ(x,s;y,t)\xi(y,t) & s<t \\
\cZ(x,s; \cdot, s) = \delta_x,
\end{cases}
\end{equation}
where the same space-time white noise $\xi$ is used for all choices of initial coordinates $(x,s)$, and $\delta_x$ is the delta mass at $x$. The initial condition is understood in the weak sense, i.e., with probability one $\lim_{t\to s}\int f(y)\cZ(x,s; y, t) \diff y = f(x)$ for all smooth functions $f$ of compact support.

This random field was constructed in \cite[Theorem 3.1]{alberts2014continuum} and is a continuous process. (More precisely, the field constructed in \cite{alberts2014continuum} satisfies \eqref{e.SHE definition} with coefficient $\frac{1}{2}$ in place of $\frac{1}{4}$, but this is related to our solution by a simple scaling by constants; we adopt this variant of the SHE in order to obtain more convenient coefficients later and match other parts of the literature.)

We will also need the multi-line extension originally introduced in \cite{o2016multi}. First denote 
\[
p_t(x)=\frac{1}{\sqrt{\pi t}}\exp(-x^2/t)
\]
and, for any $m\in\N$ and $s<t$, let 
\begin{equation}\label{e.lambda definition}
\begin{split}
\Lambda_m &= \Bigl\{(t_1, \ldots, t_m) \in \R^m: t_1 \leq  \ldots  \leq t_m\Bigr\} \quad\text{and}\\
\Lambda_m([s,t]) &= \Bigl\{(t_1, \ldots, t_m) \in \R^m: s\leq t_1 \leq  \ldots  \leq t_m \leq t\Bigr\}.
\end{split}
\end{equation}
For any $s<t$, $x, y \in \RR$, and $k\in\ZZ_+$, we let $\cZ_k(x,s;y,t)$ be defined by
\begin{align}\label{e.Z_k continuum}
\cZ_k(x,s;y,t):= p_{t-s}(y-x)^k
\left(
1+\sum_{m=1}^\infty \int_{\Lambda_m([s,t])}\int_{\RR^m}\!\!R((x_1,t_1),\ldots, (x_m,t_m)) \prod_{i=1}^m W(\diff x_i, \diff t_i)
\right),
\end{align}
where $R$ denotes the $m$-point correlation function for a collection of $k$ non-intersecting rate two Brownian bridges which all start at
 $x$ at time $s$ and end at $y$ at time $t$. This generalizes the definition \eqref{e.SHE definition} of $\cZ$ above; in particular, $\cZ_1 = \cZ$. Further, $\mc Z_k(x,s;y,t)$ can be thought of as the partition function for $k$ disjoint paths in the CDRP all starting at the common point $(x,s)$ and ending at the common point $(y,t)$.

In \cite{o2016multi} the definition \eqref{e.Z_k continuum} is given for any fixed $k$ and $x,s,y,t$ and the $L^2(W)$ convergence of the chaos expansion is proved. In \cite[Corollary 1.9, 1.11]{nica2021intermediate}, it is shown that
\[
(y,k)\mapsto \log\frac{\cZ_k(0,0;y,t)}{\cZ_{k-1}(0,0;y,t)}
\]
is a (scaled) KPZ$_t$ line ensemble, as defined in \cite[Theorem 2.15]{CH14} (up to certain scalings by constant coefficients); therefore $\cZ_k(x,s;y,t)$ has a version which is continuous in $y$, for any fixed $k$ and $x, s, t$.
In \cite[Theorem 1.1]{LW}, it is further shown that $(x,y,t)\mapsto \cZ_k(x,s;y,t)$ can be defined as a continuous function, for any fixed $k$ and $s$.
Recall also that $\cZ$ can be defined as a four-parameter random continuous function.
It is also shift, shear, and reflection invariant (in distribution) (see Lemma~\ref{l.Z symmetries}).

\medskip

\noindent\textbf{Scaling.} Under certain limiting transitions (either $t\to \infty$ or $\beta\to\infty$) and after appropriate centering and scaling (see \eqref{e.Z_1 Z_2 relation with h_1 h_2} and \eqref{e.scaling for KPZ}), the logarithm of $\cZ$, which can be understood as a solution to the KPZ equation, converges to the directed landscape \cite{AHALE,wu2023kpz}.
While we do not actually use this convergence in this paper, it is for this reason that we adopt our choice of scaling coefficients as mentioned above.\footnote{Consider $\tilde \cZ$, the solution to the usual multiplicative SHE $\partial_t \tilde \cZ = \frac{1}{2}\partial_{yy}\tilde\cZ + \tilde\cZ\xi$, and the processes $\tilde \cZ_k$ as defined in \cite[eq. (9)]{o2016multi}. Then it holds that $\cZ_k(x,s;y,t)=2^k\tilde\cZ_k(2x,2s;2y,2t)$, where the $2^k$ factor can be thought of as ensuring the initial condition, $\cZ(x,s; \cdot, t)\to \delta_x$ as $t\to s$ in the weak sense, holds.}

\subsubsection{Multi-point partition function with distinct endpoints}   \label{sss:mpp}

For any $\bx=(x_1,\cdots, x_n), \by=(y_1,\cdots, y_n)\in \Lambda_n$, and $s<t$, we define
\begin{equation}\label{e.K definition}
\cK_n(\bx, s;\by, t) := \det[\cZ(x_i,s;y_j,t)]_{i,j=1}^n.
\end{equation}
Then from the continuity of $\cZ=\cZ_1$, we have that $\cK_n(\bx, s;\by, t)$ is almost surely continuous in all the variables. 

$\cK_n(\mb x, s; \mb y, t)$ should be understood as the partition function of $n$ disjoint paths, one each from $(x_i,s)$ to $(y_i,t)$ for $i=1, \ldots, n$, where the starting points are \emph{distinct} (as otherwise the determinant above is zero, reflecting that the entropy of disjoint paths with common endpoints is zero). From this perspective \eqref{e.K definition} can be understood as a form of the Lindstr\"om-Gessel-Viennot lemma (see Lemma~\ref{lem:disj-alter-def} ahead) in the continuum. To go from $\mc K_n$ to $\mc Z_n$ (i.e., the case of common starting and ending points), one must first normalize $\mc K_n$ by a factor that matches the entropy contribution and then take a limit of the distinct points. This normalization factor is explicit and given by the product of the Vandermonde determinants of $\mb x$ and $\mb y$ (see \eqref{e.M definition} ahead).

\subsubsection{KPZ line ensemble}
An important tool in the study of the KPZ equation and the free energy of the CDRP is the KPZ$_t$ line ensemble from \cite{CH14} and its associated Gibbs property.
To be concise, at this stage we simply state  some relevant facts connecting these objects, with references given in Section~\ref{s.prelim tools}.

For any $t>0$ and $x\in\RR$, we denote 
\begin{equation}\label{e.Z_1 Z_2 relation with h_1 h_2}
\h_{t,1}(x) := \log\cZ(0,0;x,t)+\frac{t}{12},\quad \h_{t,2}(x) := \log\frac{\cZ_2(0,0;x,t)}{\cZ(0,0;x,t)}+ \frac{t}{12}.
\end{equation}
We note that $\mc H(x,t) := \h_{t,1}(x)$ solves the KPZ equation
\begin{align*}
\partial_t \mc H = \tfrac{1}{4}\partial_x^2 \mc H + \tfrac{1}{4}(\partial_x \mc H)^2 + \dot{W}
\end{align*}
in a formal sense, i.e., $\cZ(x,t) = \exp(\mc H(x,t))$ solves the multiplicative SHE \eqref{e.SHE definition}.

Thanks to the scaling in defining $\cZ$, $\h_{t,1}$ has the parabolic decay of $-x^2/t$. More precisely, $x\mapsto \smash{\fh^{\beta}_{t,1}(x)}+x^2/t$ is stationary (which can be deduced from the shear invariance of $\cZ$, Lemma~\ref{l.Z symmetries}).
We also denote
\begin{equation}\label{e.scaling for KPZ}
\hh_{t,1}(x):=t^{-1/3}\h_{t,1}(t^{2/3}x), \quad \hh_{t,2}(x):=t^{-1/3}\h_{t,2}(t^{2/3}x).
\end{equation}
As a result of this scaling, $\hh_{t,1}$ and $\hh_{t,2}$ have the parabolic decay of $-x^2$ independent of $t$. This is also the scaling under which the $t\to\infty$ limit is the parabolic Airy line ensemble.

\subsection{Main results}\label{s.results}

For any $a<b$, a Borel measurable subset $\msf A\subseteq \mc C([a,b, \R])$ is called \emph{increasing}, if for any $f\in \msf A$ and $g\in \mc C([a,b, \R])$ such that $g\geq f$ pointwise, $g\in \msf A$ holds. For a function $f$ and real number $\lambda$, we define $f-\lambda$ to be the function $x\mapsto f(x)-\lambda$.  Our first main result is the following inequality.

\begin{theorem}   \label{thm:disj-BK-line-ensemble}
There exist $C, c, L_0>0$ such that the following holds. Let $y\in\R$, $K\geq 0$, and $\msf A\subseteq \mc C([0,K],\R)$ be an increasing Borel measurable set. For any $t>0$, $L\geq L_0(t^{-1/6}\vee 1)$, and $M > C(L+y^2)^{3/4}$,
\begin{align*}
\P\Bigl(\hh_{t,2}|_{[y,y+K]} - Ct^{-1/3}\log M \in \msf A \midd \hh_{t,1}(y) = L \Bigr)
&\leq \P\left(\hh_{t,1}|_{[y,y+K]} \in \msf A\right) + 3(K+1)t^{2/3}\exp(-cM^2).
\end{align*}
The same also holds when the conditioning is replaced by $\hh_{t,1}(y) \geq L$. The same also holds under both conditionings when the processes are considered on $[y-K, y]$ rather than $[y,y+K]$.

\end{theorem}

In words, \eqref{e.airy bk} holds for the KPZ line ensemble as well if we condition on the first curve on a single point $y$ and consider the second curve on either side of $y$, up to a logarithmic shift of the second curve and an error of $\exp(-cL^2)$. The $C\log M$ term (which depends on $L$ through the lower bound on $M$) can be understood as a quantitative instantiation of the role of entropy mentioned above in Section~\ref{s.bk in positive temp}, as can be seen in the proof. While we do expect some sort of a shift to be needed for a form of the BK inequality to hold for the KPZ line ensemble, we do not know whether a logarithmic shift is the true behavior. One also obtains a version of the BK inequality for the Airy line ensemble (i.e., zero temperature) by taking, e.g., $M= Lt$ and taking $t\to\infty$; see Remark~\ref{r.other choices of error term}. The lower bounds on $L$ and $M$ are not strictly necessary, but are assumed to obtain a cleaner form of the error bound.

Theorem~\ref{thm:disj-BK-line-ensemble} is actually a fairly immediate consequence of a slightly more technical result which we state next and which will be proven in Section~\ref{s.log gamma to CDRP}. In the following, for two real-valued functions $f$, $g$ defined on a common domain $\mc D$, we write $f\geq g$ as shorthand for $f(x)\geq g(x)$ for all $x\in\mc D$. For $f$ defined on $\mc D$ and $\mc D'\subseteq \mc D$, $f|_{\mc D'}$ denotes the restriction of $f$ to $\mc D'$. Finally, a.e.\ $f \in \mc C(I, \R)$ for an interval $I$ means all functions $f$ in a probability $1$ set with respect to the law of Brownian motion on $I$ with a normally distributed starting point.

\begin{restatable}{theorem-restate}{BKlineensembletechnical}
\label{thm:disj-BK-line-ensemble-technical}
There exist $C, c>0$ such that the following holds. Let $y\in\R$, $K, R\geq 0$, and $\msf A\subseteq \mc C([0,K],\R)$ be an increasing Borel measurable set. For any $t>0$, $M > 0$, and a.e.\ $f \in \mc C([y-R,y], \R)$,
\begin{equation}\label{e.general bk}
\begin{split}
  \MoveEqLeft[14]
\P\Bigl(\hh_{t,2}|_{[y,y+K]} - Ct^{-1/3}\log M \in \msf A \midd \hh_{t,1}|_{[y-R, y]} = f \Bigr)\\
&\leq \P\left(\hh_{t,1}|_{[y,y+K]} \in \msf A\right) + \frac{3(K+1)t^{2/3}\exp(-cM^2)}{\P(\hh_{t,1}|_{[y-R,y]} \geq f)}.%
\end{split}
\end{equation}
The same also holds when the conditioning is replaced by $\hh_{t,1}|_{[y-R,y]} \geq f$, in which case we may relax the continuity assumption and allow any $f:[y-R,y]\to \R\cup\{-\infty\}$ which is upper semicontinuous. 

All of the previous also hold under both conditionings when $+K$ is replaced by $-K$ and $-R$ by $+R$ simultaneously.

\end{restatable}

Observe that the error term looks different in Theorems~\ref{thm:disj-BK-line-ensemble} and \ref{thm:disj-BK-line-ensemble-technical}, in that the latter's has a denominator which is a probability. In fact, one of the steps in going from Theorem~\ref{thm:disj-BK-line-ensemble-technical} to Theorem~\ref{thm:disj-BK-line-ensemble} is lower bounding that probability and absorbing it into the $\exp(-cM^2)$ term in the numerator; this is the reason we impose a lower bound of order $L^{3/4}$ on $M$ in Theorem~\ref{thm:disj-BK-line-ensemble} (in this setting the probability in the denominator is like $\exp(-cL^{3/2})$).

We also point out that the interval on which the values of $\hh_{t,1}$ are conditioned on and the interval on which the event of $\hh_{t,2}$ is considered are disjoint, unlike \eqref{e.airy bk} which had no such constraint. The proofs of these positive temperature results goes through the discrete log-gamma polymer model, and the above constraints arise due to the use of certain Markovian structures in the latter.

Theorem~\ref{thm:disj-BK-line-ensemble-technical} is actually proved as a consequence of the following slightly cleaner result, which is about the two-point disjoint polymer partition functions directly. It will be proven in Section~\ref{s.log gamma to CDRP}.

\begin{theorem}   \label{thm:disj-BK}
Fix any real numbers $x_1<x_2$ and $y_1\in\R$. Let $\msf A\subseteq \mc C([0,K],\R)$ for some $K>0$ be an increasing Borel measurable set. Then for any $s<t$, $R>0$, and a.e.\ $f\in \mc C([y_1-R, y_1], \R)$,
\begin{align*}
\MoveEqLeft[26]
\PP \Bigl(\log \cK_2((x_1,x_2), s; (y_1, \bm\cdot), t)|_{[y_1,y_1+K]} - \log \cZ(x_1, s; y_1, t)  \in \msf A  \midd \log \cZ(x_1, s; \bm\cdot, t)|_{[y_1-R, y_1]} = f \Bigr)\\
&\leq \PP \Bigl( \log \cZ(x_2, s; \bm\cdot, t)|_{[y_1,y_1+K]} \in \msf A \Bigr).    
\end{align*}
The same also holds when the conditioning is replaced by $\log \cZ(x_1, s; y_1, t)|_{[y_1-R, y_1]} \geq f$. The same also holds under both conditionings when $+K$ and $-R$ are simultaneously replaced by $-K$ and $+R$ in the above.
\end{theorem}

There are a number of natural generalizations of Theorem~\ref{thm:disj-BK} one can consider, such as studying $\cK_n$ rather than $\cK_2$, or allowing $x_2$ to vary as well. Some of these generalizations are in fact obtainable rather quickly by similar arguments as presented in this paper. We discuss this further in Section~\ref{s.generalizations}.

The logarithmic shift and $\exp(-cL^2)$ error terms in Theorems~\ref{thm:disj-BK-line-ensemble} and \ref{thm:disj-BK-line-ensemble-technical} arise when moving from $\log \cK((x_1,x_2), s; (y_1, \bm\cdot), t)$ to $\hh_{t,2}$, as this requires certain approximations and modulus of continuity bounds to hold. This is an instantiation of the role of entropy as the source of these error terms, since, as mentioned above, the difference between $\log \cK_2$ and $\hh_{t,2}$ is precisely that the latter (actually, $\log \cZ_2 = \hh_{t,2} + \hh_{t,1}$) is a limit of the former after $\cK_2$ is normalized by an entropy factor, the product of Vandermonde determinants of the points $\mb x$ and $\mb y$ (see \eqref{e.M definition} ahead).  See also Remark~\ref{r.spread points} ahead.

The proof of Theorem~\ref{thm:disj-BK} goes via first establishing an analogous statement for the log-gamma polymer model; this is stated ahead in Section~\ref{s.bk for log gamma} as Theorem~\ref{thm:disj-BK-log-g-new}. The log-gamma polymer model is a discrete polymer model on $\Z^2$ first introduced in \cite{MR2917766}, which we define more precisely in Section~\ref{s.log gamma}. It is an integrable model, here meaning that there are exact formulas for the joint distribution of various quantities of interest (such as the analogs of $\cZ_i$). Perhaps somewhat surprisingly, this integrability seems important for the validity of the result; indeed, there exist counterexamples to the inequality stated in Theorem~\ref{thm:disj-BK-log-g-new} for other distributions on the vertex weights; see Remark~\ref{r.counterexample}. This is in contrast to the situation in zero temperature, where  for all i.i.d.\  vertex weight distributions the BK inequality holds, in its original form for the primal environment  as described at the beginning of the paper as well as certain versions of the line ensemble form.

\begin{remark}[Terminology]
  As already indicated, the form of the inequalities we prove differ from \eqref{e.airy bk}, and in fact it is not clear what is the correct or sharpest analog of \eqref{e.airy bk} is for the KPZ line ensemble. Nevertheless, as we have already been doing, we will sometimes still refer, somewhat loosely, to the inequalities we do prove for the KPZ line ensemble (Theorems~\ref{thm:disj-BK-line-ensemble} and \ref{thm:disj-BK-line-ensemble-technical}) as well as one we prove for the continuum directed random polymer (Theorem~\ref{thm:disj-BK}) as the BK inequality. In that sense the term can be thought of as referring to a class of inequalities which, broadly, control the behavior of the second curve conditional on the first curve in terms of an unconditioned copy of the first curve.

\end{remark}

\subsection{Applications}

The BK inequality and variants of it have already proven to be very useful in studies involving line ensembles in the KPZ universality class. Here we briefly discuss its role in two works, \cite{GH22} and \cite{ganguly2023brownian}.

The first work, \cite{GH22}, studies upper tail behavior of the first curve of line ensembles satisfying certain natural assumptions like correlation inequalities by making use of probabilistic resampling ideas. The results are of two types: the first type obtains sharp upper and lower bounds on the probability of upper tail events, while the second type describes the limit shape or fluctuation behavior of the line ensembles when conditioned on such events. In fact, the latter is an ingredient of the proofs of the former. To establish the fluctuation behavior of the top curve under an upper tail conditioning, one needs to know that the conditioning does not cause the second curve to rise too high (e.g., in zero temperature, this would force the first curve to move up as well, affecting the fluctuations). The BK inequality is a natural tool which provides exactly such a guarantee, and (in fact, a much weaker version) is one of the assumptions imposed on the line ensembles for the framework of \cite{GH22} to apply; the arguments of the paper are not sensitive to the precise form of the inequality. Further, the BK inequality \eqref{e.airy bk} was proved in \cite{GH22} for the parabolic Airy line ensemble, but a (weaker) version for the KPZ line ensemble was only conjectured (in an earlier version). 

Our Theorem~\ref{thm:disj-BK-line-ensemble} is sufficient to meet the requirements of the framework developed in \cite{GH22}. Consequently, the KPZ line ensemble results, including the sharp upper tail estimates for the KPZ equation, also follow (as now formally recorded in \cite{GH22} in the latest arXiv version).
More precisely, the earlier conjectured inequality in \cite{GH22} stated that the KPZ line ensemble satisfies a form of the BK inequality with a shift that is logarithmic in the size of the interval over which the first curve is conditioned on, rather than a function of the values of the first curve as proved in Theorem~\ref{thm:disj-BK-line-ensemble-technical}. Based on the entropic source of the shift as discussed above, we believe that conjecture is too strong to be true as stated. In any case, on a practical level, Theorem~\ref{thm:disj-BK-line-ensemble-technical} would suffice for taking the place of such an inequality in most arguments (as was the case in \cite{GH22} itself).

The second work, \cite{ganguly2023brownian}, also studies upper tail behavior, but of certain path measures associated to the first curve of the line ensembles. More precisely, in zero temperature, the first curve of the parabolic Airy line ensemble records the weight of a geodesic in the directed landscape started at $(0,0)$ as the endpoint varies along a horizontal line, while in positive temperature, the first curve of the KPZ line ensemble records the free energy of a polymer starting at $(0,0)$ as its endpoint varies in the same way. In both cases we have an associated path measure, namely the law of the geodesic and the annealed polymer measure, respectively. \cite{ganguly2023brownian} established that, when conditioned on the first curve of the respective line ensembles to be large and taking the largeness parameter to $\infty$, the path measures (appropriately rescaled) in both cases converge to a Brownian bridge. Theorem~\ref{thm:disj-BK-line-ensemble} is a crucial input to the proofs, and is used in a number of locations for similar reasons as in \cite{GH22}.

We note that chronologically it is somewhat unusual that this paper appears subsequent to the above mentioned papers \cite{GH22} or \cite{ganguly2023brownian} for which it serves as an input. This is due to the fact that one of the roles of this paper is to fix an earlier incorrect formulation and proof of the BK inequality of the KPZ line ensemble that those papers relied on, which was first pointed out to us by Xuan Wu.
We also emphasize that the arguments in this paper have no dependencies that create a circular argument with the previous papers. In particular, this paper is self-contained except for the use of some basic properties and estimates for the objects of interest.

\subsection{Ideas of proofs}\label{s.iop}

As mentioned above, the first key step is to reduce \Cref{thm:disj-BK-line-ensemble-technical} to \Cref{thm:disj-BK}.
Then to get \Cref{thm:disj-BK}, we consider the log-gamma polymer model, to be introduced in \Cref{s.log gamma}.
It is a discrete model, and is known to converge to the CDRP under appropriate scaling (see \Cref{l.Z joint convergence} below).
Under the geometric Robinson-Schensted-Knuth (gRSK) correspondence (see e.g., \cite{COSZ}), 
the log-gamma polymer can be mapped to another collection of random variables, which can be viewed as a discrete line ensemble.
This is the prelimit analog of the connection between the CDRP and the KPZ line ensemble.

We note that it is more convenient for us to work with the log-gamma polymer than CDRP.
Indeed, unlike the KPZ line ensemble, the log-gamma polymer discrete line ensemble consists of finitely many random variables; and exact expressions for its joint distributions are available (see e.g., \Cref{lem:shift}). 
In particular, it also has a Markovian structure.
We will prove the log-gamma version of the BK inequality (\Cref{thm:disj-BK-log-g-new}), and then pass to a scaling limit to obtain \Cref{thm:disj-BK}.

To prove \Cref{thm:disj-BK-log-g-new} we prove an extension of a well-known identity \cite{C21,NY} of disjoint path partition functions in the log-gamma environment and that in the associated discrete line ensemble; a version of this extension in the zero temperature case was communicated to us a few years ago by Duncan Dauvergne. The original identity (and its earlier zero temperature analog \cite{DOV}) required the starting points of the paths to be on the bottom-most line and the ending points to be on the top most line. Here we relax the requirement imposed on the ending points and allow them to also be on the right boundary, at the cost of the identity requiring an additional line ensemble to be introduced (associated to the same log-gamma environment but for paths whose ending points vary in the vertical direction rather than horizontal); see Figure~\ref{f.intro extended invariance}. In a sense this extension is very natural since the gRSK correspondence is precisely a bijection when including partition functions whose endpoints are the top and right boundaries; however, the partition functions that make up the second line ensemble are not defined by endpoints varying on the right boundary but on the left boundary. 

\begin{figure}
  \begin{tikzpicture}[scale=0.7]

\newcommand{\thewidth}{4}
\newcommand{\theheight}{3}

\draw[densely dotted] (0,0) grid (\thewidth, \theheight);

\foreach \y in {0, ..., \theheight}
\foreach \x in {0,...,\thewidth}
{
  \node[fill, circle, inner sep = 1pt] at (\x,\theheight-\y) {};
}

\foreach \x in {0,...,\thewidth}
  \draw[orange, thick] (0,0) -- (\x,\theheight);

\node[anchor = south, scale=0.75] at (4, 3) {$(m, n)$}; 
\node[anchor = north, scale=0.75] at (0,0) {$(1, 1)$};

\begin{scope}[shift={(6,0)}]
  \draw[densely dotted] (0,0) grid (\thewidth, \theheight);

  \foreach \y in {0, ..., \theheight}
  \foreach \x in {0,...,\thewidth}
  {
  \node[fill, circle, inner sep = 1pt] at (\x,\theheight-\y) {};
  }

  \foreach \y in {0,...,\theheight}
    \draw[cyan, thick] (\thewidth,\theheight) -- (0,\y);

  \node[anchor = south, scale=0.75] at (4, 3) {$(m, n)$}; 
  \node[anchor = north, scale=0.75] at (0,0) {$(1, 1)$}; 
\end{scope}

%%% LINE ENSEMBLE

\begin{scope}[shift={(11,0)}]

\node[anchor = south, scale=0.75] at (5, 3) {$(m, n)$}; 
\node[anchor = south, scale=0.75] at (1, 3) {$(1, n)$}; 
\node[anchor = north east, scale=0.75] at (3.2, 1) {$(a, n-a+1)$};

\draw[red, thick] (3,1.045) -- ++(1,0) -- ++(0,0.955) -- ++(1,0) -- ++(0,1) -- ++(2,0) -- ++(0,-0.955) -- ++(1,0) -- ++(0,-1);

\draw[cyan, thick] (3,1) -- ++(2,0) -- ++(0,1) -- ++(2,0) -- ++(0,-1) -- ++(1,0);

\draw[orange, thick] (3,0.955) -- ++(5,0);

\renewcommand{\thewidth}{4}
\renewcommand{\theheight}{3}

\begin{scope}[shift={(1,0)}]
\foreach \y in {0, ..., \theheight}
\foreach \x in {\y,...,\thewidth}
{
  \ifnum \x < \thewidth
    \draw[densely dotted] (\x,\theheight-\y) -- ++(1,0);
  \fi
  \node[fill, circle, inner sep = 1pt] at (\x,\theheight-\y) {};

  \ifnum \x > \y 
    \ifnum \y < \theheight
    \draw[densely dotted] (\x,\theheight-\y) -- ++(0,-1);
    \fi
  \fi
}
\end{scope}

\foreach \y in {0, ..., \theheight}
{
  \draw[dashed, dash phase = 4pt, thick] (\thewidth+1, \y) -- ++(1,0);
   \draw[dashed, dash phase = 3pt, thick] (\thewidth+2, \y) -- ++(1,0);
  \node[draw, fill=white, circle, inner sep = 1pt] at (\thewidth+2,\y) {};
}

\renewcommand{\thewidth}{3}
\renewcommand{\theheight}{3}

\begin{scope}[shift={(7,0)}]
\foreach \y in {0, ..., \theheight}
\foreach \x in {\y,...,\thewidth}
{
  \ifnum \x < \thewidth
    \draw[densely dotted] (\thewidth-\x,\theheight-\y) -- ++(-1,0);
  \fi
  \node[fill, circle, inner sep = 1pt] at (\thewidth-\x,\theheight-\y) {};

  \ifnum \x > \y 
    \ifnum \y < \theheight
    \draw[densely dotted] (\thewidth-\x,\theheight-\y) -- ++(0,-1);
    \fi
  \fi
}

\node[anchor = south west, scale=0.75] at (-0.75, 3) {$(m+1, n)$}; 
\node[anchor = south, scale=0.75] at (3, 3) {$(m+n, n)$}; 
\node[anchor = west, scale=0.75] at (1, 1) {$(m+\ell, \ell)$};

\end{scope}
\end{scope}
\end{tikzpicture}
\caption{The left and middle panels depict the endpoints whose partition functions define the left and right line ensembles (black vertices), respectively, in the right panel. Proposition~\ref{p.extended invariance} says that the partition function in the original environment when the starting point is on the bottom line and the ending point is on the top or right boundaries equals a certain partition function in the combined line ensemble depicted on the right (some paths which contribute to the latter are shown). Note that the two line ensembles are connected by an auxiliary column of vertices shown in white; these vertices have associated weights which prevent a certain double counting from appearing in the weights of certain paths in the joint line ensemble, which is needed for Proposition~\ref{p.extended invariance}. }\label{f.intro extended invariance}
\end{figure}
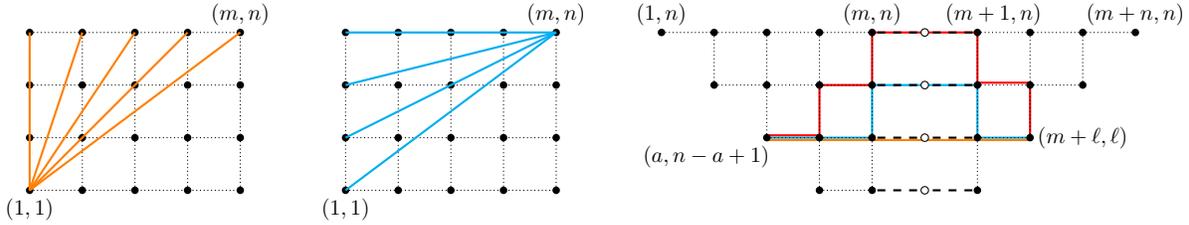

Then we work with the partition function of two disjoint log-gamma polymers. Using the above mentioned identity, we relate it to the single polymer partition function in a smaller environment. A key observation is that, conditional on the top line of the line ensemble, the latter partition function's law can be written as a reweighting of the law of the partition function in the original environment explicitly. Moreover, the reweighting has a negative effect. This gives us the inequality by invoking the FKG inequality for the underlying i.i.d.\ random variables.

\subsection{Organization of paper}
In Section~\ref{s.bk for log gamma} we introduce the log-gamma polymer model and prove the version of the BK inequality for disjoint polymers (the analog of Theorem~\ref{thm:disj-BK}) for it. In Section~\ref{s.prelim tools}, we introduce some basic estimates and properties of the main objects in the paper. These are all fairly well-known, but a couple have their proofs deferred to Appendix~\ref{s.misc proofs}. Finally, in Section~\ref{s.log gamma to CDRP}, we perform the scaling limit to obtain Theorem~\ref{thm:disj-BK} and give the argument that uses the latter to yield Theorem~\ref{thm:disj-BK-line-ensemble}.

\subsection*{Notation}
Throughout this paper, we will use $C, c>0$ to denote large and small constants, whose values may, and often will, change from line to line.
For any  $x,y\in\RR\cup\{-\infty, \infty\}$ with $x\le y$, we denote $\qq{x, y} = [x,y]\cap \ZZ$.

For any $m\in\NN$, we recall the simplex $\Lambda_m$ defined by
\[
\Lambda_m=\{(t_1, \cdots, t_m): t_1\le \cdots \le t_m\},
\]
and for any $s<t$, 
\[
\Lambda_m([s,t])=\{(t_1, \cdots, t_m): s\le t_1\le \cdots \le t_m\le t\}.
\]
We use $\rDe_m$ and $\rDe_m([s,t])$ to denote the interiors of $\Lambda_m$ and $\Lambda_m([s,t])$, respectively.

For any topological space $\mc X$, we use $\mc C(\mc X, \R)$ to denote the space of real continuous functions on $\mc X$ with the uniform topology.

We will occasionally use notation such as $\P(\,\bm\cdot \mid X = K)$ for a random element $X$ in some space and a deterministic element $K$ of the same space to denote the conditional probability distribution given that $X=K$. The precise meaning of this is to consider the regular conditional distribution $\P(\,\bm\cdot \mid X)$ (which exists in the situations we consider by well-known abstract results such as \cite[Theorem 8.5]{Kallenberg}) and evaluate the associated probability kernel at $K$. In discrete settings such as in Section~\ref{s.bk for log gamma} these expressions can be defined much more directly, of course.

\subsection*{Acknowledgements}
SG was partially supported by  NSF CAREER grant DMS-1945172 and a Miller Professorship at The Miller Institute of Basic Research in Science. MH was partially supported by NSF grants DMS-1937254 and DMS-2348156, and by the NTU Support Grant for Research Award Number \#025661-00007. LZ is supported by the NSF grant DMS-2505625.
The authors first learned a version of the extended invariance identity (Proposition~\ref{p.extended invariance},  in the zero temperature setting) from Duncan Dauvergne a few years ago, and we thank him for allowing us to state and use it in this paper.
The authors would also like to thank Amol Aggarwal and Jiaoyang Huang for some helpful discussions, and Xuan Wu for raising the question of the BK inequality in positive temperature settings.

\section{BK for disjoint endpoints in log-gamma}\label{s.bk for log gamma}

As mentioned in Section~\ref{s.iop} on the ideas of the proof, to prove Theorem~\ref{thm:disj-BK} we work with the log-gamma polymer model, which we recall in Section~\ref{s.log gamma}. For the latter, our main result in this section is Theorem~\ref{thm:disj-BK-log-g-new}, its version of the BK inequality. Another important ingredient is the extended invariance identity for partition functions in a given environment and that in the associated line ensemble, stated ahead in Section~\ref{s.extended identity} as Proposition~\ref{p.extended invariance}.

\subsection{The log-gamma polymer model}\label{s.log gamma}

The log-gamma model was introduced in \cite{Se12}, and is defined as follows. Take $\theta>0$, and let $\{X_{\mb v}\}_{\mb v\in \Z^2}$ be independent random variables with inverse-gamma distribution of parameter $\theta$; i.e., with probability density given by 
\begin{equation}\label{e.inverse gamma density}
\frac{1}{\Gamma(\theta)}x^{-\theta-1}\exp(-x^{-1}) \quad \text{for } x>0.
\end{equation}

For any $\mb u, \mb v \in \Z^2$ with $\mb u\le \mb v$ in each coordinate, and $k\in\N$, we define the $k$ disjoint polymer partition function
\begin{equation}\label{e.T definition}
T_k(\mb u,\mb v) := \sum_{(\gamma_1, \ldots, \gamma_k)} \prod_{j=1}^k \prod_{\mb w\in \gamma_j} X_{\mb w},
\end{equation}
where the summation is over all collections of $k$ \emph{vertex-disjoint} up-right paths $(\gamma_1, \ldots, \gamma_k)$, with each $\gamma_i$ from $\mb u+(i-1,0)$ to $\mb v+(-k+i,0)$. (By convention, $T_k(\mb u,\mb v):=0$ if no such collection of paths exists.)

We will often need a slight generalization of the previous definition. Take any $2k$ vertices $\mb u_1, \ldots, \mb u_k, \mb v_1, \ldots, \mb v_k \in \Z^2$, such that for each $i, j\in \intint{1, k}$, $\mb u_i\le \mb v_j$ in each coordinate, and any up-right path from $\mb u_i$ to $\mb v_j$ and any up-right path from $\mb u_j$ to $\mb v_i$ must intersect (e.g., all the $\mb u_i$ lie on the same horizontal line in order and the same for the $\mb v_j$). We define the $k$ disjoint polymer partition function
\begin{equation}\label{e.T definition-sep}
T([\mb u_1, \ldots, \mb u_k], [\mb v_1, \ldots, \mb v_k]) := \sum_{(\gamma_1, \ldots, \gamma_k)} \prod_{j=1}^k \prod_{\mb w\in \gamma_j} X_{\mb w},
\end{equation}
where the summation is over all collections of $k$ \emph{vertex-disjoint} up-right paths $(\gamma_1, \ldots, \gamma_k)$, with each $\gamma_i$ from $\mb u_i$ to $\mb v_i$. 
In particular, it is again taken to be zero if no such collection of paths exists.

Our main result for the log-gamma polymer is the following prelimiting version of Theorem~\ref{thm:disj-BK}. It will be proven in Section~\ref{s.bk log gamma proof new}.

\begin{theorem}   \label{thm:disj-BK-log-g-new}
Let $\theta>0$. For any integers $m, n\ge 2$, $f:\intint{1,m}\to \R_+$, and increasing Borel measurable set $\msf A \subseteq \R^{(m-1)(n-1)}$,
\begin{equation*}%
\begin{split}
\MoveEqLeft[18]
  \PP \left( \Bigl\{ f(m)^{-1} T([(1, 1), (a, 1)], [(m,n), (m,b)])\Bigr\}_{a\in \intint{2, m}, b\in \intint{1, n-1}} \in \msf A  \;\Big|\; T_1((1,1), (\bm\cdot,n))|_{\intint{1,m}} = f \right) \\ 
  &\le \PP \left( \Bigl\{T_1((a-1,1), (m-1,b))\Bigr\}_{a\in \intint{2, m}, b\in \intint{1, n-1}} \in \msf A \right).    
\end{split}
\end{equation*}
\end{theorem}

\begin{remark}\label{r.spread points}
  We point out that Theorem~\ref{thm:disj-BK-log-g-new} holds when the pair of starting points and the pair of ending points are adjacent (i.e., $a=2$ and $b = n-1$), in which case 
  $$T([(1, 1), (a, 1)], [(m,n), (m,b)])/T_1((1, 1), (m,n))$$
  is precisely the second curve of the discrete log-gamma line ensemble (associated to system size $n$). Thus one might expect to be able to take a limit of Theorem~\ref{thm:disj-BK-log-g-new} in this special case directly to the KPZ line ensemble and obtain the BK inequality without the logarithmic shift and probability error term present in Theorem~\ref{thm:disj-BK-line-ensemble-technical}.

  However, the issue is that the righthand side of Theorem~\ref{thm:disj-BK-log-g-new} involves the first curve of the log-gamma diffusion associated to system size $n-1$. Typically, due to entropy considerations, the first curve of the $(n-1)$-system line ensemble is $\log n$ higher than the second curve of the $n$-system. More precisely, this is essentially because the probability of two random walks of length $n$ started at a unit order separation have probability of order $n^{-1/2}$ of remaining disjoint for their lifetimes. This logarithmic shift is seen on a technical level in the fact that the centering terms required to take the limit to the KPZ line ensemble differs by $\log n$ between the two system sizes (as can be read off of Lemma~\ref{l.Z joint convergence} ahead). As a result, the limiting inequality obtained is trivial.

The crucial aspect of Theorem~\ref{thm:disj-BK-log-g-new} is that it allows points to be separate, which avoid this issue. The strategy is thus to take the limit when the endpoints are separate on the diffusive scale (so the underlying random walks have unit order probability of remaining disjoint) and to make use of modulus of continuity properties in the limit in order to make the points coincide. This will be done in Sections~\ref{s.prelim tools} and \ref{s.log gamma to CDRP}.
\end{remark}

Notice that Theorem~\ref{thm:disj-BK-log-g-new} allows the second endpoint to vary on the vertical line $\{(m,b) : b\in\intint{1,n-1}\}$. In the scaling limit to the CDRP, this corresponds to the endpoint varying temporally. However, it will be more convenient to allow the point to vary spatially, as these partition functions are what is captured by the KPZ line ensemble. For this, we record the following simple corollary that instead allows the second endpoint to vary on the horizontal line of index $n-1$.

\begin{corollary}\label{c.bk coupling log gamma new}
Let $\theta>0$. For any integers $b, m, n\ge 2$ with $b < m$ and function $f: \intint{1,b}\to\R$, the following holds. Let $\{\tilde T([(1, 1), (a, 1)], [(b,n), (b',n-1)])\}_{a\in \intint{2, b}, b'\in \intint{b+1, m}}$ be sampled from the conditional law of 
$$\{T([(1, 1), (a, 1)], [(b,n), (b',n-1)])\}_{a\in \intint{2, b}, b'\in \intint{b+1, m}} \text{ given } T_1((1,1), (\bm\cdot,n))|_{\intint{1,b}} = f.$$
Then, there exists a coupling of the law of $\{\tilde T([(1, 1), (a, 1)], [(b,n), (b',n-1)])\}_{a\in \intint{2, b}, b'\in \intint{b+1, m}}$ and that of $\{T_1((a,1), (b',n-1))\}_{a\in \intint{2, b}, b'\in \intint{b+1, m}}$ such that, for all $a\in \intint{1, b}$ and $b'\in \intint{b+1, m}$,
\begin{align*}
 f(b)^{-1}\tilde T([(1, 1), (a, 1)], [(b,n), (b',n-1)]) \leq T_1((a-1,1), (b'-1,n-1)).
\end{align*}
\end{corollary}

\begin{proof}
The proof idea is to first use the coupling from Theorem~\ref{thm:disj-BK-log-g-new} to obtain a pair of environments such that the associated partition functions on the right boundary have the correct marginal distribution (one conditional on $T_1$ and the other not) and are dominated in the way specified there. Then, one augments both environments (i.e., extending the coupling) by inserting columns to the right of $x=m$ of the same i.i.d.\ inverse gamma random variables in both systems. This will imply a domination of the partition functions to the $(n-1)$\textsuperscript{st} line by a recursion. The details are below.

By Theorem~\ref{thm:disj-BK-log-g-new} and Strassen's theorem on stochastic domination (see, e.g., \cite{lindvall1999strassen}), it follows that there exists a coupling between $\{\tilde T([(1, 1), (a, 1)], [(b,n), (b,\ell)])\}_{a\in \intint{2, b}, \ell\in \intint{1, n-1}}$ and $\{T_1((a-1,1), (b-1,\ell))\}_{a\in \intint{2, b}, \ell\in \intint{1, n-1}}$ such that, almost surely, for all $a\in \intint{2, b}, \ell\in \intint{1, n-1}$,
\begin{align}\label{e.conditioned inequality}
f(b)^{-1}\tilde T([(1, 1), (a, 1)], [(b,n), (b,\ell)]) \leq T_1((a-1,1), (b-1,\ell)).
\end{align}
Next consider an i.i.d.\  collection $\{X_{(i,j)}\}_{i\in\intint{b+1, m}, j\in\intint{1,n-1}}$ of inverse gamma random variables. For $b'\in \intint{b+1, m}$, define
\begin{align*}
\tilde \tau([(1, 1), (a, 1)], [(b,n), (b',n-1)]) &:= \sum_{\ell=1}^{n-1} \tilde T([(1, 1), (a, 1)], [(b,n), (b,\ell)])\cdot \!\!\!\!\!\!\!\! \sum_{\gamma: (b+1,\ell) \to (b',n-1)} \prod_{(i,j)\in\gamma} X_{(i,j)},\\
\tau_1((a-1,1), (b'-1,n-1)) &:= \sum_{\ell=1}^{n-1} T_1((a-1,1), (b-1,\ell))\cdot \sum_{\gamma: (b+1,\ell) \to (b',n-1)} \prod_{(i,j)\in\gamma} X_{(i,j)}.
\end{align*}
Observe that, by \eqref{e.conditioned inequality},
\begin{align}\label{e.tau inequality}
f(b)^{-1}\tilde\tau([(1, 1), (a, 1)], [(b,n), (b',n-1)]) \leq \tau_1((a-1,1), (b'-1,n-1))
\end{align}
holds for all $a\in \intint{2, b}$ and $b'\in \intint{b+1, m}$.

It follows immediately from the definition \eqref{e.T definition-sep} of $T$ that $\{\tilde \tau([(1, 1), (a, 1)], [(b,n), (b',n-1)])\}_{a\in\intint{2,b}, b'\in\intint{b+1,m}}$ has the law of 
$$\{T([(1, 1), (a, 1)], [(b,n), (b',n-1)])\}_{a\in\intint{2,b}, b'\in\intint{b+1,m}} \text{ conditioned on } T_1((1,1), (\bm\cdot,n))|_{\intint{1,b}} = f.$$
 Similarly, $\tau_1((a-1,1), (b'-1,n-1))_{a\in\intint{2,b}, b'\in\intint{b+1,m}}$ has the law of $\{T_1((a-1,1), (b'-1,n-1))\}_{a\in\intint{2,b}, b'\in\intint{b+1,m}}$. 
This along with \eqref{e.tau inequality} completes the proof.
\end{proof}

\begin{remark}
Note that Corollary~\ref{c.bk coupling log gamma new} only allows the second endpoint to vary horizontally to the right of the last point $(m,n)$ whose associated partition function value is conditioned on. This is the source of the analogous ordering of intervals present in Theorems~\ref{thm:disj-BK-line-ensemble-technical} and \ref{thm:disj-BK}.
\end{remark}

\begin{remark}[Non-validity of Theorem~\ref{thm:disj-BK-log-g-new} for Bernoulli distributions]\label{r.counterexample}
We point out that Theorem~\ref{thm:disj-BK-log-g-new} is not true for general distributions, and thus the integrability of the log-gamma polymer model seems to play a crucial role in its validity. Indeed, consider the case of $m=n=2$ with $\{X_{\mb v}\}_{\mb v\in\intint{1,2}^2}$ being i.i.d.\ Bernoulli($p$) random variables for any $p\in(0,1)$ (i.e., $\P(X_{\mb v} = 1) = 1-\P(X_{\mb v} = 0) = p$). For notational convenience, write $\smash{T^{(2)}_1} = T_1((1,1), (2,2))$, $\smash{T^{(1)}_1} = T_1((1,1), (1,1))$ (the superscripts indicating the system size), and $\smash{T_2^{(2)}} = T_2((1,1), (2,2) )$. Then observe that 
\begin{align*}
T^{(1)}_1 &= X_{(1,1)},\\
T_1^{(2)} &= X_{(1,1)}X_{(2,2)}\left(X_{(1,2)}+X_{(2,1)}\right),\quad\text{and}\\
T_2^{(2)} &= X_{(1,1)}X_{(1,2)}X_{(2,1)}X_{(2,2)}.
\end{align*}
If we condition on $T_1^{(2)} = 2$, we force $X_{(1,1)}=X_{(1,2)}=X_{(2,1)}=X_{(2,2)} = 1$. So $T_2^{(2)}/T_1^{(2)} = \frac{1}{2}$ almost surely. Thus it holds that
\begin{align*}
\P\left(2^{-1}T_2^{(2)} \geq t \mid T_1^{(2)} = 2\right) = 1 \quad\text{ for all } t\leq \frac{1}{2}.
\end{align*}
On the other hand, for any $t>0$,
$$\P\Bigl(T_1^{(1)} \geq t\Bigr) \leq \P\Bigl(T_1^{(1)} \geq 1\Bigr)  = p < 1,$$
as $T_1^{(1)} = X_{(1,1)}$ is a Bernoulli($p$) random variable. Thus the main display of Theorem~\ref{thm:disj-BK-log-g-new} does not hold when $t\in(0,\frac{1}{2}]$, for any $p\in(0,1)$.

Though in this example there was complete freezing of all the vertex variables, this is not necessary. Indeed, one can consider vertex weights distributed as i.i.d.\ random variables uniform on $[0,1]$. Then $T^{(2)}_{1} \leq 2$ almost surely again. If one fixes $\delta\in (0,\frac{1}{2})$ and conditions on $T^{(2)}_1 \geq 2-\delta$, it follows that $\min(X_{(1,1)}, X_{(2,2)})) \geq 1-\frac{1}{3}\delta$ and $\min(X_{(1,2)}, X_{(2,1)}) \geq 1-\frac{2}{3}\delta$ almost surely. So, under the same conditioning, $T^{(2)}_2 \in [1-2\delta, 1]$ almost surely, since $(1-\frac{1}{3}\delta)^2(1-\frac{2}{3}\delta)^2 \geq 1-2\delta$. Thus,
\begin{align*}
\P\left(\frac{T_2^{(2)}}{T^{(2)}_1} \geq t \midd T_1^{(2)} \geq 2-\delta\right) = 1 \quad\text{ for all } t\leq \frac{1}{2}-\delta.
\end{align*}
But $\P(T_1^{(1)} \geq t) = 1-t < 1$ for all $t\in(0,1]$, yielding that the inequality of Theorem~\ref{thm:disj-BK-log-g-new} does not hold whenever $t\in(0,\frac{1}{2}-\delta]$. Thus violations to the inequality also exist among continuous distribution and in the absence of complete freezing.

\end{remark}

\subsection{Geometric RSK for general polymer models on $\Z^2$}

In this section we recall some basic facts about the geometric RSK (gRSK) correspondence in the context of general polymer models on $\Z^2$ that will be needed in the proof of Theorem~\ref{thm:disj-BK-log-g-new}. In particular, in this section as well as in Section~\ref{s.extended identity}, we work in the setting where $\{X_{\mb v}\}_{\mb v \in \Z^2}$ is simply a collection of non-negative real numbers. Quantities like $T_k$ are defined as in \eqref{e.T definition} and \eqref{e.T definition-sep}.

The gRSK correspondence can be defined via either row/column insertion algorithms (see e.g., \cite[Section 2]{COSZ}), or disjoint paths. We take the latter one, which involves defining certain partition functions $Z_j$ by
\begin{align}\label{e.Z_i defintion}
Z_1(\mb u,\mb v) := T_1(\mb u,\mb v),\text{ and } Z_j(\mb u,\mb v) := \frac{T_j(\mb u,\mb v)}{T_{j-1}(\mb u,\mb v)} \text{ for $j\ge 2$}
\end{align}
(with the convention of taking $Z_j(\mb u,\mb v) := 0$ when $T_{j-1}(\mb u,\mb v)=0$).

\begin{remark}\label{r.Z_i vs mathcal Z_i}
  We note that $Z_j$ for $j\geq 2$ is not quite a discrete analog of $\mc Z_j$ as defined in \eqref{e.Z_k continuum}. The difference is that the former is the ratio of partition functions for $j$ and $j-1$ disjoint polymers while the latter \emph{is} the partition function for $j$ disjoint polymers. This discrepancy arises from differing notation in two parts of the literature (\cite{COSZ} for $Z_j$ and \cite{o2016multi} for $\cZ_j$) and we have chosen to be consistent with both.
\end{remark}

For a fixed starting point such as $(1,1)$, we view $\smash{\{Z_j\}_{j=1}^n}$ with the ending point varying as a line ensemble, with $Z_j$ the values of the $j$\textsuperscript{th} line. Indeed, taking appropriate scaling limits of this line ensemble results in the KPZ or parabolic Airy line ensembles (see \cite{wu2019tightness} along with \cite{dimitrov2022characterization}, and \cite[Corollary 25.2]{AHALE}, respectively).

Disjoint polymer partition functions can be expressed as a determinant, a fact which will be very useful for us. The first equality in the following is an immediate consequence of the Lindstr\"om-Gessel-Viennot lemma \cite[Corollary 2]{gessel1989determinants} and the second is by definition \eqref{e.Z_i defintion}.

\begin{lemma}\label{lem:disj-alter-def}
Suppose $(\mb u_1, \ldots, \mb u_k)$ and $(\mb v_1, \ldots, \mb v_k)$ satisfy the condition stated before \eqref{e.T definition}. Then,
\[
T([\mb u_1, \ldots, \mb u_k], [\mb v_1, \ldots, \mb v_k]) = \det\left(T_1(\mb u_i, \mb v_j)\right)_{i,j=1}^k = \det\left(Z_1(\mb u_i, \mb v_j)\right)_{i,j=1}^k.
\]
\end{lemma}

We mention that \cite{gessel1989determinants} works in a setting in which weights are assigned to edges rather than vertices, but the statement nevertheless applies to our setting by an appropriate procedure of assigning the vertex weights to the edges; see the brief discussion in \cite[Section 2.2.4]{C21} for details. We also mention that \cite[Corollary 2]{gessel1989determinants} applies to all directed acyclic graphs, expressing multi-path disjoint partition functions as determinants of matrices whose entries are given by single path partition functions.

\subsection{An extended invariance identity}\label{s.extended identity}
We work on the domain $\intint{1,m}\times\intint{1,n}$. Here, one should think of $m$ as being much larger than $n$, though this is not formally imposed or needed.

For the convenience of notation, denote $T_0((1,1), (m,n))=1$, and $(1,1)^0=(m,n)^0=\emptyset$, and $(1,1)^k= [(1,1), \ldots, (k,1)]$, $(m,n)^k=[(m,n), \ldots, (m,n-k+1)]$, for each $k\in \N$. 

Recall that Proposition~\ref{thm:invnyc} equated the disjoint polymer partition functions $T$ to partition functions in the line ensemble $Y$, under the condition that the starting points were on the bottom line and the ending points on the top line. The aim of this section is an alternate representation for the partition function where the endpoint is \emph{not} on the top line, also in terms of partition functions in the line ensemble. This is Proposition~\ref{p.extended invariance}, which we state just ahead after introducing some notation. This result and its proof hold deterministically for all non-negative vertex weights, not just i.i.d.\ inverse-gamma weights.

\subsubsection{Notation} We introduce some useful notation.
Define the index sets $J_1[m,n]$, $J_2[m,n]$, and $J[m,n]$ by (see Figure~\ref{f.J sets})
\begin{equation}\label{e.J definition}
\begin{split}
   J_1[m,n] &:= \left\{(i, j) \in \Z^2 : 1\le j \le \min\{i,n\}, 1\le i\le m\right\},\\
 J_2[m,n] &:= \left\{(i, j) \in \Z^2 : 1\le j \le \min\{m+n+1-i,m\}, m+1\le i\le m+n\right\}, \quad\text{and}\\
 J[m,n] &:= J_1[m,n] \cup J_2[m,n].
\end{split}
\end{equation}

Now, for any $m,n \in \N$, and each $(i,j)\in J_1[m,n]$, denote (see Figure~\ref{f.Z definition})
\begin{equation}\label{e.Z^m,n first}
Z^{m,n}_j(i) = Z_j((1, 1), (i, n)),
\end{equation}
and for each $(i,j)\in J_2[m,n]$, denote
\begin{equation}\label{e.Z^m,n second}
Z^{m,n}_j(i) = Z_j((1, i-m), (m, n)).
\end{equation}
These are, respectively, the line ensembles associated to the partition functions indicated in the left and right panels of Figure~\ref{f.Z definition}.

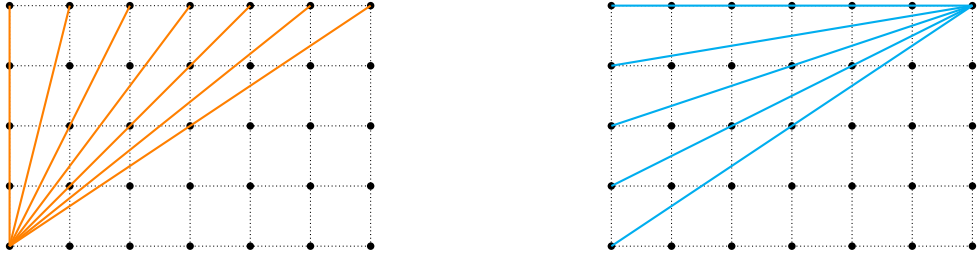
\begin{figure}
\begin{tikzpicture}[scale=0.8]
\newcommand{\thewidth}{6}
\newcommand{\theheight}{4}

\draw[densely dotted] (0,0) grid (\thewidth, \theheight);

\foreach \y in {0, ..., \theheight}
\foreach \x in {0,...,\thewidth}
{
  \node[fill, circle, inner sep = 1pt] at (\x,\theheight-\y) {};
}

\foreach \x in {0,...,\thewidth}
  \draw[orange, thick] (0,0) -- (\x,\theheight);

\begin{scope}[shift={(10,0)}]
  \draw[densely dotted] (0,0) grid (\thewidth, \theheight);

  \foreach \y in {0, ..., \theheight}
  \foreach \x in {0,...,\thewidth}
  {
  \node[fill, circle, inner sep = 1pt] at (\x,\theheight-\y) {};
  }

  \foreach \y in {0,...,\theheight}
    \draw[cyan, thick] (\thewidth,\theheight) -- (0,\y);
\end{scope}
\end{tikzpicture}
\caption{Left: the orange lines connect the common starting point and varying ending points of the paths whose partition functions determine the values of $Z^{m,n}_j(i)$ for $(i,j)\in J_1[m,n]$, and therefore also of $\{Y_{\mb v}\}_{\mb v\in V_1[m,n]}$. Right: the blue lines connect the varying starting points and common ending point of the paths whose partition functions determine the values of $Z^{m,n}_j(i)$ for $(i,j)\in J_2[m,n]$, and therefore also of $\{Y_{\mb v}\}_{\mb v\in V_2[m,n]}$. Note that in both cases the partition functions of the paths joining $(1,1)$ and $(m,n)$ are included, which implies that $Z^{m,n}_j(m) = Z^{m,n}_j(m+1)$ for all $j\in\intint{1,n}$.}\label{f.Z definition}
\end{figure}

We wish to associate weights $Y$ to the vertices of the line ensembles defined by $Z^{m,n}$, but there is a minor annoyance in the indexing. Namely, $Z^{m,n}_j$'s values should be associated to the $j$\textsuperscript{th} line from the top of the line ensemble, which in the usual coordinates has $y$-coordinate $n-j+1$. To streamline the presentation, we define $V_1[m,n]$ and $V_2[m,n]$ ($V$ for vertices) by
\begin{align*}
V_1[m,n] &= \Bigl\{(i,j): (i,n+1-j)\in J_1[m,n]\Bigr\} \quad\text{and}\\
V_2[m,n] &= \Bigl\{(i,j): (i,n+1-j)\in J_2[m,n]\Bigr\}.
\end{align*}

For each $(i, j)\in V_1[m,n]$, let
\begin{equation}\label{e.Y definition 2}
Y_{(i,j)} := \frac{Z_{n+1-j}^{m,n}(i)}{Z_{n+1-j}^{m,n}(i-1)}
\end{equation}
where we use the convention that $Z_j^{m,n}(0)=1$.  Next, for each $(i,j) \in V_2[m,n]$, let
\begin{equation}\label{e.Y definition 2 ext}
Y_{(i,j)} := \frac{Z_{n+1-j}^{m,n}(i)}{Z_{n+1-j}^{m,n}(i+1)}.
\end{equation}
Thus $Y|_{V_1[m,n]}$ are the multiplicative increments of the line ensemble associated to the partition functions indicated on the left panel of Figure~\ref{f.Z definition}, while $Y|_{V_2[m,n]}$ is that associated to the right panel.

\subsubsection{Invariance identity for endpoints on top line}

In this subsection we state the previously known invariance identity for when the endpoints are on the top line, before stating the extended version in Section~\ref{s.partition functions and invariance identity}. We first define a partition function in the line ensemble $Y$ restricted to $V_1[m,n]$. 

For any $k\in\N$ and any $\mb u_1, \ldots,  \mb u_k, \mb v_1, \ldots,  \mb v_k \in V_1[m,n]$ satisfying the condition above \eqref{e.T definition-sep}, we define
\begin{equation}\label{e.S definition}
S\bigl([\mb u_1, \ldots, \mb u_k], [\mb v_1, \ldots,  \mb v_k]\bigr) := \sum_{(\gamma_1, \ldots, \gamma_k)} \prod_{\mb w\in \cup_{i=1}^k\gamma_i} Y_{\mb w},
\end{equation}
where the summation is over all tuples of disjoint up-right paths $(\gamma_1, \ldots, \gamma_k)$, with $\gamma_i$ from $\mb u_i$ to $\mb v_i$ for $i\in\intint{1,k}$. Since the $Y$ are multiplicative increments of $Z$, the above definition is a partition function in the environment given by $Y$ and induced by the line ensemble, analogous to \eqref{e.T definition-sep}.

It will be convenient to write, for $x\in\N$, $(x,1)^{\shortuparrow} = (x,\max\{n-x+1, 1\})$. See Figure~\ref{f.coords}.

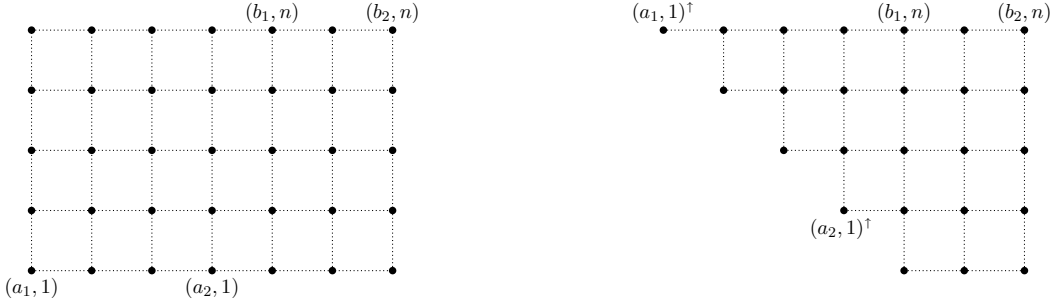
\begin{figure}[t]
\begin{tikzpicture}[scale=0.8]
\draw[densely dotted] (0,0) grid (6, 4);

\newcommand{\thewidth}{6}
\newcommand{\theheight}{4}

\foreach \y in {0, ..., \theheight}
\foreach \x in {0,...,\thewidth}
{
  \node[fill, circle, inner sep = 1pt] at (\x,\theheight-\y) {};
}

\node[scale=0.7, anchor=north] at (0,0) {$(a_1,1)$};
\node[scale=0.7, anchor=north] at (3,0) {$(a_2,1)$};
\node[scale=0.7, anchor=south] at (4,4) {$(b_1,n)$};
\node[scale=0.7, anchor=south] at (6,4) {$(b_2,n)$};

\begin{scope}[shift={(10.5,0)}]
\foreach \y in {0, ..., \theheight}
\foreach \x in {\y,...,\thewidth}
{
  \ifnum \x < \thewidth
    \draw[densely dotted] (\x,\theheight-\y) -- ++(1,0);
  \fi
  \node[fill, circle, inner sep = 1pt] at (\x,\theheight-\y) {};

  \ifnum \x > \y 
    \ifnum \y < \theheight
    \draw[densely dotted] (\x,\theheight-\y) -- ++(0,-1);
    \fi
  \fi
}

\node[scale=0.7, anchor=south] at (0,4) {$(a_1, 1)^{\shortuparrow}$};
\node[scale=0.7, anchor=north east] at (3.6,1) {$(a_2, 1)^{\shortuparrow}$};
\node[scale=0.7, anchor=south] at (4,4) {$(b_1,n)$};
\node[scale=0.7, anchor=south] at (6,4) {$(b_2,n)$};
\end{scope}
\end{tikzpicture}
\caption{A depiction of the coordinates in Proposition~\ref{thm:invnyc}.}\label{f.coords}
\end{figure}

\begin{proposition}[\protect{\cite[Theorem 1.7]{NY}, \cite[Theorem 1.1]{C21}}]   \label{thm:invnyc}
Fix any set of non-negative weights $\{X_{\mb v}\}_{\mb v\in \Z^2}$ and define $T$ and $S$ as in \eqref{e.T definition-sep} and \eqref{e.S definition} with these weights. Then for any $k\in\N$, integers $1\le a_1 < \ldots <a_k < b_1 <  \ldots  < b_k$, and $n\ge 2$, deterministically, 
\begin{align*}
T\bigl([(a_1, 1), \ldots , (a_k, 1)], [(b_1, n), \ldots, (b_k, n)]\bigr)
&= S\bigl([(a_1, 1)^{\shortuparrow}, \ldots , (a_k, 1)^{\shortuparrow}], [(b_1, n), \ldots,  (b_k, n)]\bigr).
\end{align*}
\end{proposition}

Next we move to giving some setup to state the extended invariance identity. We will need to introduce a mildly larger graph which connects the vertex sets $V_1[m,n]$ and $V_2[m,n]$ and their associated line ensembles, which we do in the next subsection.

\subsubsection{Underlying graph}

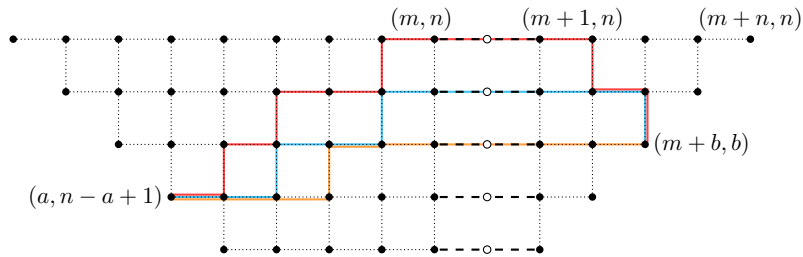
\begin{figure}[b]
\begin{tikzpicture}[scale=0.7]

\node[anchor = south east, scale=0.8] at (9.5, 3) {$(m, n)$}; 
\node[anchor = east, scale=0.8] at (4, 0) {$(a, n-a+1)$};

\draw[red, thick, opacity=0.75] (4,0.045) -- ++(1,0) -- ++(0,0.955) -- ++(1,0) -- ++(0,1) -- ++(2,0) -- ++(0,1) -- ++(3,0);
\draw[red, thick, opacity=0.75] (11,3) -- ++(1,0) -- ++(0,-0.955) -- ++(1.045,0) -- ++(0,-1);

\draw[cyan, thick, opacity=0.75] (4,0) -- ++(2,0) -- ++(0,1) -- ++(2,0) -- ++(0,1) -- ++(3,0);
\draw[cyan, thick, opacity=0.75] (11,2) -- ++(2,0) -- ++(0,-0.955);

\draw[orange, thick, opacity=0.75] (4,-0.045) -- ++(3,0) -- ++(0,1) -- ++(1,0) -- ++(0,0.045) -- ++(3,0);
\draw[orange, thick, opacity=0.75] (11,1) -- ++(2,0);

\newcommand{\thewidth}{8}
\newcommand{\theheight}{4}

\begin{scope}[shift={(1,-1)}]
\foreach \y in {0, ..., \theheight}
\foreach \x in {\y,...,\thewidth}
{
  \ifnum \x < \thewidth
    \draw[densely dotted] (\x,\theheight-\y) -- ++(1,0);
  \fi
  \node[fill, circle, inner sep = 1pt] at (\x,\theheight-\y) {};

  \ifnum \x > \y 
    \ifnum \y < \theheight
    \draw[densely dotted] (\x,\theheight-\y) -- ++(0,-1);
    \fi
  \fi
}
\end{scope}

\foreach \y in {0, ..., \theheight}
{
  \draw[dashed, dash phase = 4pt, thick] (9, \y-1) -- ++(1,0);
   \draw[dashed, dash phase = 3pt, thick] (10, \y-1) -- ++(1,0);
  \node[draw, fill=white, circle, inner sep = 1pt] at (10,\y-1) {};
}

\renewcommand{\thewidth}{4}
\renewcommand{\theheight}{4}

\begin{scope}[shift={(11,-1)}]
\foreach \y in {0, ..., \theheight}
\foreach \x in {\y,...,\thewidth}
{
  \ifnum \x < \thewidth
    \draw[densely dotted] (\thewidth-\x,\theheight-\y) -- ++(-1,0);
  \fi
  \node[fill, circle, inner sep = 1pt] at (\thewidth-\x,\theheight-\y) {};

  \ifnum \x > \y 
    \ifnum \y < \theheight
    \draw[densely dotted] (\thewidth-\x,\theheight-\y) -- ++(0,-1);
    \fi
  \fi
}

\node[anchor = south west, scale=0.8] at (-0.5, 4) {$(m+1, n)$}; 
\node[anchor = south, scale=0.8] at (4, 4) {$(m+n, n)$}; 
\node[anchor = west, scale=0.8] at (2, 2) {$(m+b, b)$};

\end{scope}
\end{tikzpicture}
\caption{A depiction of the graph underlying the line ensembles. Up to swapping the $y$-coordinate $j$ by $n+1-j$, these also depict the coordinates associated to the index sets $J_1[m,n]$ and $J_2[m,n]$, respectively. The paths in red, blue, and orange are examples of the paths contributing to the partition function $S([(a,1)^{\shortuparrow}], [(m+b,b)])$; they go up-right in the left line ensemble, and down-right in the right line ensemble.}\label{f.J sets}
\end{figure}

To combine the line ensembles associated to $V_1[m,n]$ and $V_2[m,n]$ as just mentioned, we introduce a vertical line of auxiliary vertices with associated weights (at this stage this may seem artificial, but it will make the extended invariance identity Proposition~\ref{p.extended invariance} cleaner to state); these are the vertices denoted by unfilled circles in Figure~\ref{f.J sets}. We label these vertices as $(m+\frac{1}{2}, j)$ for $j\in\intint{1,n}$ and the weight $Y_{(m+\frac{1}{2},j)}$ associated to $(m+\frac{1}{2},j)$ is given by
\begin{equation}\label{e.Y new vertex weight}
Y_{(m+\frac{1}{2},j)} = Z^{m,n}_{n+1-j}(m)^{-1}.
\end{equation}

Next we specify the edges and their orientation in the graph underlying the line ensemble $Y$, which explains the joining of the two that was mentioned just above. See Figure~\ref{f.J sets} for a depiction. In words, the edges in the left line ensemble are directed right and up, while the edges in the right line ensemble are directed right and \emph{down}, and there is a right-directed edge connecting $(m,j)$ with $(m+\frac{1}{2},j)$ as well as $(m+\frac{1}{2},j)$ with $(m+1,j)$.

More formally, for $(i,j)\in V_1[m,n]$, $(i,j)$ is connected with $(i+1,j)$ and $(i,j+1)$ in that direction, so long as the corresponding latter endpoint is a member of $V_1[m,n]$. Each $(i,j)\in V_2[m,n]$ is connected with $(i+1,j)$ and $(i,j-1)$ in that direction, so long as the corresponding latter endpoint is a member of $V_2[m,n]$. Finally, we connect $(m,j)$ with $(m+\frac{1}{2},j)$ and $(m+\frac{1}{2},j)$ with $(m+1,j)$, in the mentioned direction in both cases.

\subsubsection{Partition functions}\label{s.partition functions and invariance identity}
With this preparation we may define partition functions in this combined line ensemble and then give the extended invariance identity, Proposition~\ref{p.extended invariance}. First, by paths we simply mean paths in the directed graph just specified. Denote the vertex set of this graph by 
$$V[m,n] := V_1[m,n]\cup V_2[m,n] \cup \left\{(m+\tfrac{1}{2}, j) : j\in\intint{1,\min(m,n)}\right\}.$$
Then for any $k\in\N$ and any $\mb u_1, \ldots,  \mb u_k, \mb v_1, \ldots,  \mb v_k \in V$, define
\begin{align}\label{e.S definition extended}
S([\mb u_1, \ldots, \mb u_k], [\mb v_1, \ldots, \mb v_k]) := \sum_{(\gamma_1, \ldots, \gamma_k)} \prod_{\mb w\in \cup_{i=1}^k\gamma_i} Y_{\mb w},
\end{align}
with the value of the empty sum set to 0 by convention. 
Note that this indeed agrees with the definition of $S$ given in \eqref{e.S definition} when $\mb u_i, \mb v_i \in V_1[m,n]$.

\begin{remark}\label{r.properties of paths in joint line ensemble}
We make a few simple observations about the nature of paths in this joint line ensemble to familiarize the reader; see Figure~\ref{f.J sets}. By definition, they go up-right in the left part of the ensemble, and down-right in the right part. As a result, the path cannot cross from the left to the right at a line below that of the left or right endpoints. Finally, to move from the left side to the right side, it must necessarily pass through all three of $(m,j)$, $(m+\frac{1}{2}, j)$, and $(m+1, j)$ for a unique $j\in\intint{1,\min(m,n)}$, as there are no vertical edges in the vertical line $x=m+\frac{1}{2}$.
\end{remark}
Next we state the extended invariance identity, which (in the zero temperature setting) we first learned from Duncan Dauvergne through private communication. For $j\in\intint{1,n}$, define $(m,j)^{\shortrightarrow} = (m+j,j)$. Recall also that for $i\in\intint{1,m}$, $(i,1)^\shortuparrow = (i,n-i+1)$.

\begin{proposition}[Extended invariance identity]\label{p.extended invariance}

  Let $k\in\N$ and $\ell\in\N\cup\{0\}$ with $0\leq \ell\leq k$ (here, $\ell$ will be the number of endpoints on the top boundary). Let $1\leq a_1 < \ldots < a_k \leq m$, $1\leq b_1 <  \ldots < b_\ell \leq m$, $n > b_{\ell+1} >  \ldots  > b_k \geq 1$. Then
  \begin{align*}
  \MoveEqLeft[8]
  T\left([(a_1,1), \ldots, (a_k,1)], [(b_1, n), \ldots, (b_\ell, n), (m, b_{\ell+1}), \ldots , (m, b_{k})]\right)\\
  &= S\left([(a_1,1)^{\shortuparrow}, \ldots, (a_k,1)^{\shortuparrow}], [(b_1, n), \ldots, (b_\ell, n), (m, b_{\ell+1})^{\shortrightarrow}, \ldots , (m, b_{k})^{\shortrightarrow}]\right).
  \end{align*}
\end{proposition}

The conditions on $a_i$ and $b_i$ are merely to ensure planar ordering of the points in the primal environment so as to ensure that the partition function is non-zero. Observe that the same condition also ensures the planar ordering of the corresponding points in the line ensemble, thus guaranteeing that that partition function is also non-zero.

Before turning to the proof of Proposition~\ref{p.extended invariance}, we record a quick consequence of it that will be important for our overall arguments towards Theorem~\ref{thm:disj-BK-log-g-new}. To state it, for any integers $1\le a \le m$ and $1\le b \le n$, we define a function $F^{m,n}_{a,b}:\R_+^{J[m,n]}\to \R_+$, by
\begin{equation}\label{e.F definition}
\begin{split}
\MoveEqLeft[1]
F^{m,n}_{a,b}(\{Z^{m,n}_j(i)\}_{(i,j)\in J[m,n]}) := S([(a,1)^{\shortuparrow}], [(m, b)^{\shortrightarrow}]).
\end{split}
\end{equation}

\begin{corollary}  \label{corr:cross-line}
For any integers $1\le a \le m$ and $1\le b \le n$, 
\begin{equation}\label{e.new identity j=1}
  T_1((a,1), (m,b)) = F^{m,n}_{a,b}(\{Z^{m,n}_j(i)\}_{(i,j)\in J[m,n]}),
\end{equation}
and for any $2 \le a\le m$ and $1\le b\le n-1$,
\begin{equation}\label{e.new identity j=2}
\frac{T([(1,1), (a,1)], [(m,n), (m,b)])}{T_1((1,1), (m,n))} = F^{m-1,n-1}_{a-1,b}(\{Z^{m,n}_{j+1}(i+1)\}_{(i,j)\in J[m-1,n-1]}).
\end{equation}
\end{corollary}

Observe that the \eqref{e.new identity j=2} asserts that the lefthand side, a ratio of the partition function of two disjoint paths with that of a single one, is a single-path partition function in a smaller line ensemble obtained by removing the top line of the original line ensemble. Our argument for Theorem~\ref{thm:disj-BK-log-g-new} will make crucial use of this observation after writing out the relation of the laws of these two line ensembles.

\begin{proof}[Proof of Corollary~\ref{corr:cross-line}]
First, \eqref{e.new identity j=1} is simply a restatement of the $k=1$ case of Proposition~\ref{p.extended invariance}. Next, again by Proposition~\ref{p.extended invariance},
\begin{align*}
T([(1,1), (a,1)], [(m,n), (m,b)]) = S([(1,1)^{\shortuparrow}, (a,1)^{\shortuparrow}], [(m,n), (m,b)^{\shortrightarrow}]).
\end{align*}
Observe that on the righthand side, there is a single path from $(1,n)$ to $(m,n)$. So the righthand side factorizes as 
$$S([(1,1)^{\shortuparrow}], [(m,n)])\cdot \sum_{k=2}^{\min\{a,n-b+1\}} S([(a,1)^{\shortuparrow}],[(m,n-k+1)])\cdot Z^{m,n}_{n-k+1}(m)^{-1}\cdot S([(m,n-k+1)], [(m,b)^{\shortrightarrow}]),$$
where the sum is precisely the partition function of a single path in $S$ from $(a,1)^{\shortuparrow}$ to $(m,b)^{\shortrightarrow}$ which cannot use the top line. This yields, by an application of Proposition~\ref{p.extended invariance} again, that 
\begin{align*}
\MoveEqLeft[2]
\frac{T([(1,1), (a,1)], [(m,n), (m,b)])}{T_1((1,1), (m,n))}\\
&= \sum_{k=2}^{\min\{a,n-b+1\}} S([(a,1)^{\shortuparrow}],[(m,n-k+1)])\cdot Z^{m,n}_{n-k+1}(m)^{-1}\cdot S([(m,n-k+1)], [(m,b)^{\shortrightarrow}]).
\end{align*}
Now, this can be thought of as the partition function of a single path in the smaller line ensemble consisting of only the bottom $n-1$ lines. We claim that it in fact exactly equals $F^{m-1,n-1}_{a-1,b}(\{Z^{m,n}_{j+1}(i+1)\}_{(i,j)\in J[m-1,n-1]})$. The proof is mainly keeping track of indices. More precisely, we first make the substitution $k\mapsto k+1$ in the previous display, along with the definition of $S$ from \eqref{e.S definition extended}, to write it as 
\begin{align}
\MoveEqLeft[14]
\sum_{k=1}^{\min\{a-1,n-b\}} S([(a,1)^{\shortuparrow}],[(m,n-k)])\cdot Z^{m,n}_{n-k}(m)^{-1}\cdot S([(m,n-k)], [(m,b)^{\shortrightarrow}])\label{e.intermediate to equal F}\\
&= \sum_{k=1}^{\min\{a-1,n-b\}} \left(\sum_{\gamma_k^+} \prod_{\mb v\in\gamma_k^+} Y_{\mb v}\right)\cdot Z^{m,n}_{n-k}(m)^{-1}\cdot \left(\sum_{\gamma_k^-} \prod_{\mb v\in\gamma_k^-} Y_{\mb v}\right),\nonumber
\end{align}
where $\gamma_k^+$ is summed over all up-right paths in the line ensemble from $(a,1)^\shortuparrow = (a,n-a+1)$ to $(m,n-k)$ and $\gamma_k^-$ is summed over all down-right paths in the line ensemble from $(m+1,n-k)$ to $(m,b)^{\shortrightarrow} = (m+b,b)$.

We will now show that \eqref{e.intermediate to equal F} equals $F^{m-1,n-1}_{a-1,b}(\{Z^{m,n}_{j+1}(i+1)\}_{(i,j)\in J[m-1,n-1]})$ by keeping track of the indices carefully. For this, first recall from \eqref{e.Y definition 2} and \eqref{e.Y definition 2 ext} that $Y_{(i,j)} = Z^{m,n}_{n-j+1}(i)/Z^{m,n}_{n-j+1}(i-1)$ when $(i,j)\in V_1[m,n]$ and $Y_{(i,j)} = Z^{m,n}_{n-j+1}(i)/Z^{m,n}_{n-j+1}(i+1)$ when $(i,j)\in V_2[m,n]$. Thus replacing instances of $Z^{m,n}_{n-j+1}(i)$ with $Z^{m,n}_{n-j+2}(i+1)$ and then (only in the subscript) replacing $n$ by $n-1$ amounts to replacing $Y_{\mb v}$ by $Y_{\mb v+(1,0)}$. Making use of this fact in the definition \eqref{e.F definition} of $F^{m,n}_{a,b}$ shows that 
\begin{align*}
\MoveEqLeft[8]
F^{m-1,n-1}_{a-1,b}\left(\{Z^{m,n}_{j+1}(i+1)\}_{(i,j)\in J[m-1,n-1]}\right)\\
&= \sum_{k=1}^{\min\{a-1,n-b\}} \left(\sum_{\tilde\gamma_k^+} \prod_{\mb v\in\tilde\gamma_k^+} Y_{\mb v +(1,0)}\right)\cdot Z^{m,n}_{n-k}(m)^{-1}\cdot \left(\sum_{\tilde\gamma_k^-} \prod_{\mb v\in\tilde\gamma_k^-} Y_{\mb v + (1,0)}\right),
\end{align*}
where $\tilde\gamma_k^+$ is summed over all up-right paths in the line ensemble from $(a-1,n-a+1)$ to $(m-1,n-k)$ and $\tilde\gamma_k^-$ is summed over all down-right paths in the line ensemble from $(m,n-k)$ to $(m-1+b,b)$. Doing the substitution $\mb v\mapsto \mb v - (1,0)$ in both products in the previous display then yields that it equals \eqref{e.intermediate to equal F}, completing the proof.
\end{proof}

Now we turn to the proof of Proposition~\ref{p.extended invariance}. Recall that $\ell$ is the number of the $k$ paths whose endpoints lie on the top boundary. The proof will reduce the proposition to the $k=1, \ell=0$ case, i.e., the case where we consider a single path's partition function where the endpoint is on the right boundary. We isolate that case as the following lemma, which we will prove in Section~\ref{s.extended invariance simple case} after giving the proof of Proposition~\ref{p.extended invariance} assuming it.

\begin{lemma}\label{l.extended invariance one path}
Let $1\leq a \leq m$, $1\leq b \leq n$. Then
  \begin{align*}
  T\left([(a,1)], [(m, b)]\right)
  = S\left([(a,1)^{\shortuparrow}, [(m, b)^{\shortrightarrow}]\right).
  \end{align*}
\end{lemma}

\begin{proof}[Proof of Proposition~\ref{p.extended invariance}]
We first observe that it suffices to prove the case of $k=1$. Indeed, by the Lindstr\"om-Gessel-Viennot lemma \cite[Corollary 2]{gessel1989determinants}, each sides of the desired equality can be expressed in the case of $k\geq 2$ as determinants of matrices whose entries are given by single path partition functions in the corresponding environments. If we establish the $k=1$ case, then the entries of these two matrices coincide and thus their determinants do as well.

When $k=1$, $\ell = 0$ or $1$. The case of $\ell = 1$ is addressed by Proposition~\ref{thm:invnyc}. So it only remains to handle the case of $\ell=0$, i.e., the case where the endpoint lies on the right boundary. This is precisely Lemma~\ref{l.extended invariance one path}.
\end{proof}

\subsubsection{The extended invariance in the $k=1,\ell=0$ case} \label{s.extended invariance simple case}
Here we give the proof of Lemma~\ref{l.extended invariance one path}. The main step is the following lemma, whose proof will be given in Section~\ref{s.proof of intermediate identity}.

\begin{lemma}   \label{lem:cross-line}
For any integers $1\le a \le m$ and $1\le b \le n$,
\begin{equation}   \label{eq:cl1}
T_1((a,1), (m,b)) = \sum_{k=1}^{\min\{a, n-b+1\}} \frac{T([(1,1)^{k-1}, (a,1)],[(m,n)^k])\cdot T([(1,1)^k],[(m,n)^{k-1}, (m,b)])}{T_k((1,1), (m,n))\cdot T_{k-1}((1,1), (m,n))}.
\end{equation}
Moreover, for any $1\leq j \le a\le m$ and $1\le b\le n-j+1$,
\begin{multline}   \label{eq:cl2}
\frac{T([(1,1)^{j-1}, (a,1)], [(m,n)^{j-1}, (m,b)])}{T_{j-1}((1,1), (m,n))} \\ = \sum_{k=j}^{\min\{a, n-b+1\}} \frac{T([(1,1)^{k-1}, (a,1)],[(m,n)^k])\cdot T([(1,1)^k],[(m,n)^{k-1}, (m,b)])}{T_k((1,1), (m,n))\cdot T_{k-1}((1,1), (m,n))}.
\end{multline}
\end{lemma}

\begin{remark}
In a simple case, the terms appearing in Lemma~\ref{lem:cross-line} can be understood directly. Indeed, consider \eqref{eq:cl1} with $a=2$. Observe that the product of the partition functions $T_1((2,1), (m,b))$ and $T_1((1,1), (m,n))$ is the partition function of all pairs of paths such that the first is from $(2,1)$ to $(m,b)$ and the second is from $(1,1)$ to $(m,n)$. This partition function can be broken into two terms: the case where the two paths remain disjoint, and the case where they intersect. The partition function restricted to the first case is precisely $T([(1,1), (2,1)], [(m,n), (m,b)])$. In the second case, we can reroute the paths to obtain two new paths, one which connects $(1,1)$ and $(m,b)$ and the second which connects $(2,1)$ and $(m,n)$. Thus the second term equals $T_1((1,1), (m,b))\cdot T_1((2,1), (m,n))$. Putting it together we obtain,
\begin{align*}
\MoveEqLeft[10]
T_1((2,1), (m,b)) \cdot T_1((1,1), (m,n))\\
&= T([(1,1), (2,1)], [(m,n), (m,b)]) + T_1((1,1), (m,b))\cdot T_1((2,1), (m,n)),
\end{align*}
which is equivalent to the determinant expression from Lemma~\ref{lem:disj-alter-def}. Now,  dividing both sides by $T_1((1,1), (m,n))$ yields precisely \eqref{eq:cl1} in this particular case. The proof of Lemma~\ref{lem:cross-line} in Section~\ref{s.proof of intermediate identity} will proceed by an analysis of larger determinants.
\end{remark}
\begin{proof}[Proof of Lemma~\ref{l.extended invariance one path}]
The basic idea of the proof is to show that the $k$\textsuperscript{th} summand in \eqref{eq:cl1} is exactly the contribution to $S([(a,1), (m,b)])$ coming from the paths that cross from the left line ensemble to the right one on the $k$\textsuperscript{th} line from the top. Observe that a priori we have no handle on the values of such partition functions unless $k=1$, in which case we can apply the single-path case of Proposition~\ref{thm:invnyc}. The strategy to handle larger $k$ is to pack the top $k-1$ lines with paths (and divide by the partition function of those $k-1$ paths, which is the weight of a single configuration) and then invoke the general $k$ case of Proposition~\ref{thm:invnyc}; this will essentially equate the $k$\textsuperscript{th} of \eqref{eq:cl1} to the mentioned partition function. Observe that this packing indeed ensures that the endpoints in the line ensemble are the same as would arise from $k$ disjoint paths ending at the corner $(m,n)$, thus making Proposition~\ref{thm:invnyc} applicable. 

Now we turn to the details, starting with some notation. 
For $x, k\in\N$, let $(x,1)^{k, \shortuparrow} = [(x,n), (x+1,n-1), \ldots, (x+k-1,n-k+1)]$ and recall $(m,n)^{k} = [(m,n), (m,n-1), \ldots ,(m, n-k+1)]$. Now, looking at the first factor in the $k$\textsuperscript{th} summand of \eqref{eq:cl1}, we observe that by \Cref{thm:invnyc}, for any $k\in\intint{1, \min\{a, n-b+1\}}$, 

\begin{align*}
\frac{T([(1,1)^{k-1}, (a,1)],[(m,n)^k])}{T_{k-1}((1,1), (m,n))}
&= \frac{S([(1,1)^{k-1, \shortuparrow}, (a,1)^{\shortuparrow}],[(m,n)^{k}])}{S([(1,1)^{k-1, \shortuparrow}],[(m,n)^{k-1}]}.
\end{align*}

\begin{figure}[b]
\begin{tikzpicture}[scale=0.7]

\node[anchor = south west, scale=0.8] at (9, 3) {$(m, n)$}; 
\node[anchor = west, scale=0.8] at (9, 1) {$(m, n-k+1)$}; 
\node[anchor = east, scale=0.8] at (4, 0) {$(a, 1)^{\shortuparrow}$}; 
\node[anchor = east, scale=0.8] at (1, 3) {$(1, 1)^{\shortuparrow}$};

\draw[orange, thick] (4,0) -- ++(3,0) -- ++(0,1) -- ++(2,0);

\newcommand{\thewidth}{8}
\newcommand{\theheight}{4}

\foreach \x in {1, 2}
{
  \draw[draw=Mulberry, thick] (\x, \theheight-\x) -- ++(\thewidth+1-\x,0); 
}

\begin{scope}[shift={(1,-1)}]
\foreach \y in {0, ..., \theheight}
\foreach \x in {\y,...,\thewidth}
{
  \ifnum \x < \thewidth
    \draw[densely dotted] (\x,\theheight-\y) -- ++(1,0);
  \fi
  \node[fill, circle, inner sep = 1pt] at (\x,\theheight-\y) {};

  \ifnum \x > \y 
    \ifnum \y < \theheight
    \draw[densely dotted] (\x,\theheight-\y) -- ++(0,-1);
    \fi
  \fi
}
\end{scope}

\end{tikzpicture}
\caption{The purple paths are the only ones which contribute to $S([(1,1)^{k-1,\shortuparrow}], [(m,n)]^{k-1})$ (here, $k=3$). In $S([(1,1)^{k-1,\shortuparrow}, (a,1)^{\shortuparrow}], (m,n)^{k})$, the same purple paths are frozen, and there is an additional factor of the partition function associated to the starting and ending points $(a,1)^{\shortuparrow}$ and $(m,n-k+1)$ (the weight of the depicted orange path is one contribution to it). Thus in the ratio of the two terms the contribution of the top $k-1$ frozen paths cancels.}\label{f.cancellation}
\end{figure}
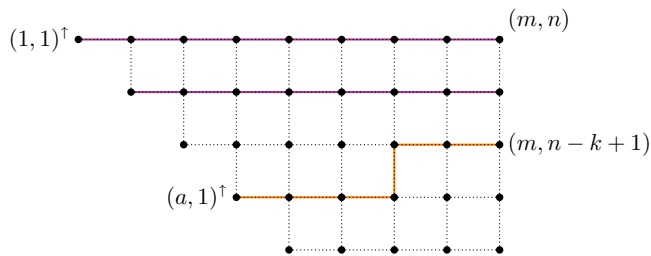

Now, for each $k$, the only contribution to $S([(1,1)^{k-1, \shortuparrow}], [(m,n)^{k-1}])$ is from  $k-1$ paths moving straight across from $(i,n-i+1)$ to $(m, n-i+1)$ for $i=1, \ldots, k-1$, as no other valid paths exist; see Figure~\ref{f.cancellation}. For analogous reasons, $S([(1,1)^{k-1, \shortuparrow}, (a,1)^\shortuparrow],[(m,n)^{k}])$ is given by the path weights of $k-1$ paths moving straight across from $(i,n-i+1)$ to $(m, n-i+1)$ for $i=1, \ldots, k-1$ multiplied by the partition function of a single polymer from $(a, 1)^{\shortuparrow}$ to $(m,n-k+1)$. Thus, the dependence on the top $k-1$ lines of $S$ factors out in both the numerator and denominator and is identical in the two. As a result, this dependence gets cancelled. This yields that
\begin{align*}
\frac{T([(1,1)^{k-1}, (a,1)],[(m,n)^k])}{T_{k-1}((1,1), (m,n))} = S\left([(a,1)^{\shortuparrow}], [(m,n-k+1)]\right).%
\end{align*}
Similarly, this time using Proposition~\ref{thm:invnyc} after a rotation to handle partition functions in the right line ensemble alone,
\begin{align*}
\frac{T([(1,1)^k],[(m,n)^{k-1}, (m,b)])}{T_{k-1}((1,1), (m,n))} = S\left([(m+1,n-k+1)], [(m, b)^{\shortrightarrow}]\right).
\end{align*}
Recall also from \eqref{e.Z_i defintion} that $Z_k^{m,n}(m) = T_k((1,1), (m,n))/T_{k-1}((1,1), (m,n))$. Substituting this equality along with the previous two displays into the first expression \eqref{eq:cl1} from \Cref{lem:cross-line} yields that
\begin{align*}
\MoveEqLeft[2]
T_1((a,1), (m,b))\\
&= \sum_{k=1}^{\min\{a, n-b+1\}}  S\left([(a,1)^{\shortuparrow}], [(m,n-k+1)]\right) \cdot Z^{m,n}_k(m)^{-1} \cdot S\left([(m+1,n-k+1)], [(m, b)^{\shortrightarrow}]\right).
\end{align*}
Recalling from \eqref{e.Y new vertex weight} that $Y_{(m+\frac{1}{2}, n-k+1)} = Z^{m,n}_k(m)^{-1}$, by Remark~\ref{r.properties of paths in joint line ensemble}, we see that the previous display is precisely the partition function $S\left([(a,1)^{\shortuparrow}], [(m, b)^{\shortrightarrow}]\right)$ decomposed based on the index of the line that the path cross from the left line ensemble to the right. This completes the proof.
\end{proof}

\subsubsection{Proof of intermediate identity}\label{s.proof of intermediate identity}

Here we give the proof of Lemma~\ref{lem:cross-line}.

\begin{proof}[Proof of Lemma~\ref{lem:cross-line}]
Take $k\in \intint{1, \min\{a, n-b+1\}}$.
Denote $\mb u_i = (i,1)$ and $\mb v_i = (m, n-i+1)$ for each $i\in\intint{1, k}$, and $\mb u_{k+1}=(a,1)$, $\mb v_{k+1}=(m, n-b+1)$. 
We consider the matrix $M=(M_{i,j})_{i,j=1}^{k+1}$, where $M_{i,j}=T_1(\mb u_i, \mb v_j)$.
Using \Cref{lem:disj-alter-def}, we have
\[
T([(1,1)^k, (a,1)], [(m,n)^k, (m,b)]) = \det M.
\]
If $k<\min\{a, n-b+1\}$, the lefthand side above is strictly positive, thus $M$ is invertible. Now consider $T([(1,1)^{k-1}, (a,1)],[(m,n)^k])$. Again by \Cref{lem:disj-alter-def}, this equals a minor of $M$ (obtained by deleting the $k$\textsuperscript{th} row and $(k+1)$\textsuperscript{st} column). Using Cramer's rule to evaluate this minor, we obtain
\begin{align*}
T([(1,1)^{k-1}, (a,1)],[(m,n)^k]) &= -(M^{-1})_{k,k+1}\cdot\det M.
\end{align*}
Similarly, we obtain
\begin{align*}
T([(1,1)^k],[(m,n)^{k-1}, (m,b)]) &= -(M^{-1})_{k+1,k}\cdot\det M,\\
T([(1,1)^{k-1}, (a,1)], [(m,n)^{k-1}, (m,b)]) &= (M^{-1})_{k,k}\cdot\det M,\\
T_k((1,1), (m,n)) &= (M^{-1})_{k+1,k+1}\cdot\det M,\\
T_{k-1}((1,1), (m,n)) &= \det\left((M^{-1})_{i,j}\right)_{i,j=k}^{k+1}\cdot\det M.
\end{align*}
Expanding $\det\left((M^{-1})_{i,j}\right)_{i,j=k}^{k+1}$ and substituting from the remaining equations above yields 
\begin{align*}
\MoveEqLeft[6]
T\left([(1,1)^k, (a,1)], [(m,n)^k, (m,b)]\right)\cdot T_{k-1}((1,1), (m,n)) \\ 
&= T([(1,1)^{k-1}, (a,1)], [(m,n)^{k-1}, (m,b)])\cdot T_k((1,1), (m,n))\\
&\qquad - T([(1,1)^{k-1}, (a,1)],[(m,n)^k])\cdot T([(1,1)^k],[(m,n)^{k-1}, (m,b)]);
\end{align*}
this is also an instance of the Desnanot-Jacobi identity (which we have essentially reproved above) applied to $M$.
By rearranging the terms we get
\begin{multline}  \label{eq:rearrange-k}
\frac{T([(1,1)^{k-1}, (a,1)],[(m,n)^k])\cdot T([(1,1)^k],[(m,n)^{k-1}, (m,b)])}{T_k((1,1), (m,n))\cdot T_{k-1}((1,1), (m,n))} \\ = 
\frac{T([(1,1)^{k-1}, (a,1)], [(m,n)^{k-1}, (m,b)])}{T_{k-1}((1,1), (m,n))} - \frac{T([(1,1)^k, (a,1)], [(m,n)^k, (m,b)])}{T_k((1,1), (m,n))}.
\end{multline}
If $k=\min\{a, n-b+1\}$, the second term in the right-hand side equals zero, as $(1,1)^k$ overlaps with $(a,1)$ or $(m,n)^k$ overlaps with $(m,b)$, so there is no space for the paths to be disjoint. Next, for the first term, either $[(1,1)^{k-1}, (a,1)]=[(1,1)^k]$ or $[(m,n)^{k-1}, (m,b)]=[(m,n)^k]$; in either case the lefthand side equals the first term on the righthand side. Thus \eqref{eq:rearrange-k} still holds.

Finally, by summing over $k\in \intint{1, \min\{a, n-b+1\}}$, or $k\in \intint{j, \min\{a, n-b+1\}}$, we get \eqref{eq:cl1} or \eqref{eq:cl2}, respectively (recalling that the second term on the righthand side in \eqref{eq:rearrange-k} is zero when $k=\min\{a,n-b+1\}$).
\end{proof}

In the remainder of Section~\ref{s.bk for log gamma} we return to working with the log-gamma polymer model rather than a general polymer model.

\subsection{Density formulas for the log-gamma polymer}

Recall the definitions of the index sets $J[m,n]$, $J_1[m,n]$, and $J_2[m,n]$ from \eqref{e.J definition} and of $Z_j^{m,n}$ from \eqref{e.Z^m,n first} and \eqref{e.Z^m,n second}. In this subsection and the next, these quantities are associated with the case where the underlying i.i.d.\ noise is given by inverse-gamma($\theta)$ random variables (recall \eqref{e.inverse gamma density}).

As mentioned, the log-gamma model is an integrable polymer model. Here, this means that we have explicit formulas for the joint distribution of a family of partition functions. We state this next, which is a consequence of results from  \cite{COSZ}. It will be proved just below.

\begin{lemma}  \label{lem:shift}
The joint probability density of $\{Z^{m,n}_j(i)\}_{(i,j)\in J[m,n]}$ at $\{z_j(i)\}_{(i,j)\in J[m,n]}$ equals
\begin{equation} \label{eq:densj}
\begin{split}
    \frac{1}{\mathfrak{Z}_{m,n}}\exp\Bigg( - \sum_{(i,j)\in J[m,n]; i\le m-1} \frac{z_j(i)}{z_j(i+1)} - \sum_{(i,j)\in J[m,n]; i\le m-1; j\le n-1} \frac{z_{j+1}(i+1)}{z_j(i)} \Bigg) \\
    \times \exp\Bigg( -\sum_{(i,j)\in J[m,n]; i\ge m+2} \frac{z_j(i)}{z_j(i-1)} - \sum_{(i,j)\in J[m,n]; i\ge m+2; j\le n-1} \frac{z_{j+1}(i-1)}{z_j(i)} \Bigg) \\ \times \prod_{(i,j)\in J[m,n]} z_j(i)^{-1} \prod_{j=1}^{\min\{m,n\}}z_j(m)^{-\theta},
\end{split}
\end{equation}
under the constraint $z_j(m) = z_j(m+1)$ for $j\in\intint{1,\min(m,n)}$, where $\mathfrak{Z}_{m,n}$ is the renormalization constant.
\end{lemma}

Before giving the proof of Lemma~\ref{lem:shift}, we make a
 simple but important observation about the explicit density function given in \eqref{eq:densj}, thus relating the law of $\{Z^{m,n}_j(i)\}_{(i,j)\in J[m,n]}$ with that of $\{Z^{m-1,n-1}_j(i)\}_{(i,j)\in J[m-1,n-1]}$. 

 We first note that $J[m-1,n-1]$ is naturally embedded inside $J[m,n]\setminus\{(i,j) \in J[m,n]: j=1\}$ and that therefore we can separate out the terms corresponding to entries of $J[m-1,n-1]$ in \eqref{eq:densj} to see the relation between the laws of $\{Z^{m,n}_j(i)\}_{(i,j)\in J[m,n]}$ and $\{Z^{m,n}_j(i)\}_{(i,j)\in J[m-1,n-1]}$. This also allows us to do so after conditioning on the values of $Z^{m,n}_1(\bm\cdot)$.

For integers $m, n \ge 2$ and any $g:\intint{1,m+n}\to \R$, we define $\Gamma_g: \R_+^{J[m-1,n-1]} \to \R_+$ as
\begin{equation}\label{e.Gamma_g definition}
\Gamma_g(\{z_j(i)\}_{(i,j)\in J[m-1,n-1]}) := \frac{1}{\tilde{\mathfrak{Z}}_{m,n,g} }  \exp\left( -\sum_{i=1}^{m-1}\frac{z_1(i)}{g(i)}-\sum_{i=m}^{m+n-2} \frac{z_1(i)}{g(i+2)} \right),
\end{equation}
where $\tilde{\mathfrak{Z}}_{m,n,g}$ is taken such that $\E\big[\Gamma_g(\{Z^{m-1,n-1}_j(i)\}_{(i,j)\in J[m-1,n-1]})\big]=1$.

Performing the separation of the terms corresponding to $J[m-1,n-1]$ after conditioning on the values of $Z^{m,n}_1(\mb\cdot)$ leads to the following statement, whose proof is by inspection and omitted.

\begin{corollary}   \label{cor:reweight}
For $g:\intint{1,m+n}\to \R$ such that $g(m) = g(m+1)$, the law of $\{Z^{m,n}_{j+1}(i+1)\}_{(i,j)\in J[m-1,n-1]}$, conditional on $Z^{m,n}_1(\bm\cdot)=g$, equals the (unconditioned) law of $\{Z^{m-1,n-1}_{j}(i)\}_{(i,j)\in J[m-1,n-1]}$ reweighted by $\Gamma_g(\{Z^{m-1,n-1}_{j}(i)\}_{(i,j)\in J[m-1,n-1]})$.  
\end{corollary}

\begin{proof}[Proof of Lemma~\ref{lem:shift}]
Denote $\hat{Z}^{m,n}_j(i) = Z^{m,n}_j(i) = Z_j((1, 1), (i, n))$ for each $(i,j)\in J_1[m,n]$, and $\hat{Z}^{m,n}_j(i) = Z_j((1, 1), (m, m+n+1-i))$ for each $(i,j) \in J_2[m,n]$. Note that $\hat Z^{m,n}_j(m) = \hat Z^{m,n}_j(m+1)$  for each $j\in\intint{1,\min(m,n)}$, just as is the case for $Z_j^{m,n}$.

We claim that the joint probability density of $\{\hat{Z}^{m,n}_j(i)\}_{(i,j)\in J[m,n]}$ at $\{z_j(i)\}_{(i,j)\in J[m,n]}$ equals \eqref{eq:densj}.
Indeed, according to \cite[Proposition 2.5]{COSZ}, $\{\hat{Z}^{m,n}_j(i)\}_{(i,j)\in J[m,n]}$ can also be obtained via a row/column insertion procedure.
Such procedure gives a Markovian structure, as stated in \cite[Theorem 3.7]{COSZ}. Further, \cite[Proposition 5.3]{COSZ} guarantees that a certain limit of this Markovian object under a specific sequence of initial conditions  coincides with $\{\hat{Z}^{m,n}_j(i)\}_{(i,j)\in J[m,n]}$.

With these facts, the claim follows by repeatedly applying the Markovian structure in \cite[Theorem 3.7(i)]{COSZ}, using \cite[Theorem 3.7(ii)]{COSZ}, and then taking the limit from \cite[Proposition 5.3]{COSZ}. 

Finally, by a form of shift-invariance \cite[Theorems 1.2 and 1.10]{DauS}, it quickly follows that $\{Z^{m,n}_j(i)\}_{(i,j)\in J[m,n]}$ has the same distribution as $\{\hat{Z}^{m,n}_j(i)\}_{(i,j)\in J[m,n]}$. In more detail, for each $i$, shift invariance allows us to shift down the endpoints $(1,i)$ and $(m,n)$ of the partition function $Z^{m,n}_j(m+i)$ to instead be $(1,1)$ and $(m,n-i+1)$, without changing the overall joint distribution, i.e., $\{Z^{m,n}_j(i)\}_{(i,j)\in J[m,n]}$ has the same distribution before and after the shifts. Pictorially, we shift all the blue paths in the right panel of Figure~\ref{f.coords} down so that the left endpoint is $(1,1)$.
Thus the conclusion follows.
\end{proof}

\subsection{BK inequality for log-gamma polymers}\label{s.bk log gamma proof new}

In this section we give the proof of Theorem~\ref{thm:disj-BK-log-g-new}. We start with a negative association statement.
To state it, recall the definition of $\Gamma_g$ from \eqref{e.Gamma_g definition}.

\begin{lemma}\label{lem:neg-cor-fA}
Let $\msf A$ be an increasing event of $\{T_1((a,1), (m-1,b))\}_{a\in \intint{1, m-1}, b\in \intint{1, n-1}}$.
For any $g:\llbracket 1, m+n-1\rrbracket \to \R_+$,
the random variable
\begin{equation}  \label{eq:gammarfA}
\Gamma_g\left(\{Z^{m-1,n-1}_j(i)\}_{(i,j)\in J[m-1,n-1]}\right)
\end{equation}
and $\msf A$ are negatively associated; i.e., \eqref{eq:gammarfA} conditional on $\msf A$ is stochastically dominated by (unconditioned) \eqref{eq:gammarfA}.
\end{lemma}

\begin{proof}
Consider both $\Gamma_g\left(\{Z^{m-1,n-1}_j(i)\}_{(i,j)\in J[m-1,n-1]}\right)$ and $\mathds{1}[\msf A]$ as functions of the i.i.d.~random variables $\{X_{i,j}\}_{i\in \intint{1,m-1}, j\in\intint{1,n-1}}$. 
Since $T_1(\mb u, \mb v)$ is increasing in $\{X_{i,j}\}_{i\in \intint{1,m-1}, j\in\intint{1,n-1}}$, for any $\mb u\le \mb v$ in each coordinate, we have that $\Gamma_g\left(\{Z^{m-1,n-1}_j(i)\}_{(i,j)\in J[m-1,n-1]}\right)$ is a decreasing function (recall \eqref{e.Gamma_g definition}) and $\mathds{1}[\msf A]$ is an increasing function. 
Therefore the conclusion follows from the FKG inequality.
\end{proof}

\begin{proof}[Proof of \Cref{thm:disj-BK-log-g-new}]
It suffices to show that, for any $g:\intint{1,m+n-1}\to \R_+$ such that $g|_{\intint{1,m}}=f$,
\begin{multline}  \label{eq:freexp2}
\PP \left( \left\{ \frac{T([(1, 1), (a, 1)], [(m,n), (m,b)])}{T_1((1,1), (m,n))} \right\}_{a\in \intint{2, m}, b\in \intint{1, n-1}} \in \msf A  \;\Big|\; Z_1^{m,n} = g \right) \\ \le \PP \left( \{T_1((a-1,1), (m-1,b))\}_{a\in \intint{2, m}, b\in \intint{1, n-1}} \in \msf A \right).    
\end{multline}
By \Cref{corr:cross-line}, the left-hand side can be written as
\[
\PP \left( \left\{ F^{m-1,n-1}_{a-1,b}(\{Z^{m,n}_{j+1}(i+1)\}_{(i,j)\in J[m-1,n-1]}) \right\}_{a\in \intint{2, m}, b\in \intint{1, n-1}} \in \msf A  \;\Big|\; Z_1^{m,n} = g \right). \]
By \Cref{cor:reweight}, this further equals
\begin{multline*}
\E \bigg( \don\left\{ \left\{ F^{m-1,n-1}_{a-1,b}(\{Z^{m-1,n-1}_{j}(i)\}_{(i,j)\in J[m-1,n-1]}) \right\}_{a\in \intint{2, m}, b\in \intint{1, n-1}} \in \msf A \right\} \\ \times \Gamma_g(\{Z^{m-1,n-1}_{j}(i)\}_{(i,j)\in J[m-1,n-1]}) \bigg),
\end{multline*}
Note that by \Cref{corr:cross-line} again, this is
\[
\E \bigg( \don\left\{ \{T_1((a-1,1), (m-1,b))\}_{a\in \intint{2, m}, b\in \intint{1, n-1}}  \in \msf A \right\} \Gamma_g(\{Z^{m-1,n-1}_{j}(i)\}_{(i,j)\in J[m-1,n-1]}) \bigg).
\]
By \Cref{lem:neg-cor-fA}, the previous display is upper bounded by the right-hand side of \eqref{eq:freexp2} (also using that $\E[\Gamma_g(\{Z^{m-1,n-1}_{j}(i)\}_{(i,j)\in J[m-1,n-1]})] = 1$). Thus, the conclusion follows.
\end{proof}

\section{Inputs and tools for continuum objects}\label{s.prelim tools}

Having established a form of the BK inequality for disjoint polymers in Theorem~\ref{thm:disj-BK-log-g-new}, our next task is to take the continuum limit to establish Theorem~\ref{thm:disj-BK}. In this section we collect some properties of the continuum objects that will be useful, and we give the proof of Theorem~\ref{thm:disj-BK} (as well as of Theorems~\ref{thm:disj-BK-line-ensemble} and \ref{thm:disj-BK-line-ensemble-technical}) in Section~\ref{s.log gamma to CDRP}.
We start by stating some distributional symmetries of $\mc Z$, namely shift, shear, and reflection invariances.

\begin{lemma}[Theorem~3.1, \cite{alberts2014continuum} or Proposition 2.3, \cite{AJRS}]\label{l.Z symmetries}
$\cZ$ has the same distribution as

\begin{enumerate}
    \item (Shift and shear) $(x,s;y,t)\mapsto \cZ(x+\nu s+\alpha, s+\eta; y+\nu t + \alpha, t + \eta) \exp( \nu^2(t-s) + 2\nu(y-x) )$, for any $\nu, \alpha, \eta\in \RR$; and
    \item (Reflection) $(x,s;y,t)\mapsto \cZ(-x,s;-y,t)$, and $(x,s;y,t)\mapsto \cZ(y,-t;x,-s)$.
\end{enumerate}
Further, for any disjoint time intervals $\{(s_i,t_i)\}_{i=1}^k$, the functions $\cZ(\cdot, s_i; \cdot, t_i)$ are independent. Also, with probability one $\cZ(x, s; y, t) > 0$ for all $x,y\in\R$ and $0<s<t$.
Finally, (1) and (2) also hold for $\mc K$ by \eqref{e.K definition}.
\end{lemma}

\subsection{Properties of multi-point partition functions}

Next we turn to a rescaled version $\cM_n$ of the multi-layer partition function $\cK_n$ from \eqref{e.K definition}. 

The rescaling to go from $\cK_n$ to $\cM_n$ is by the Vandermonde determinants of the endpoints, and can be regarded as an entropy factor; this is necessary to relate $\cK_n$ to $\cZ_n$ (as in \eqref{e.Z_k continuum}), where all the endpoints are equal. 
We define $\cM_n: \{(\mb x,s; \mb y; t): \mb x, \mb y\in \rDe_n, s,t\in\R, s<t\}\to\R$ by
\begin{equation}\label{e.M definition}
\cM_n(\bx, s;\by, t) := \frac{\cK_n(\bx, s;\by, t)}{\Delta(\bx)\Delta(\by)} = \det[\cZ(x_i,s;y_j,t)]_{i,j=1}^k\Delta(\bx)^{-1}\Delta(\by)^{-1}
\end{equation}
where $\Delta(\bx) = \prod_{i<j}(x_i-x_j)$. %

\subsubsection{Positivity and implications.} 
Using the Karlin-McGregor theorem, it is straightforward to deduce that $\cM$ is non-negative, as shown in \cite[Proposition 5.5]{o2016multi}.
The simultaneous strict inequality is proved in \cite[Theorem 1.4]{LW}, and also in \cite[Theorem 2.17]{AJRS} with a different method.

\begin{lemma} \label{lem:pos}
Almost surely, for any $s<t$, $n\in\NN$, and $\bx, \by\in \rDe_n$, it holds that $\cM_n(\bx,s;\by,t)>0$.
\end{lemma}

The $k=2$ case of \Cref{lem:pos} will be useful later, and we write it out explicitly: for any $s<t$, $x_1<x_2$, and $y_1<y_2$,
\begin{equation}  \label{eq:FR-quad}
\cZ(x_1,s;y_1,t)\cZ(x_2,s;y_2,t) >\cZ(x_1,s;y_2,t)\cZ(x_2,s;y_1,t).    
\end{equation}
This is a positive temperature analog of the quadrangle inequality in zero temperature (see e.g., \cite{DOV,BGH}).

This positivity yields an inequality relating $\cK_2$ and $\cZ$ which will be useful later. Here and in the remainder of the paper we adopt a slightly different notation for $\cK_2$ than in Section~\ref{s.intro}, namely, we write $\cK_2([(x_1, s), (x_2,s)]; [(y_1, t), (y_2,t)])$ rather than $\cK_2( (x_1, x_2), s; (y_1,y_2), t)$. This is in order to be more consistent with the notation from Section~\ref{s.bk for log gamma} and as it is slightly clearer when dealing with just two points.

\begin{lemma}\label{l.K Z ordered inequality}
The following holds almost surely. For any $x_1\leq x_2$, $y_1\leq y_2\leq y_3$, and $s<t$,
\begin{align*}
\frac{\mc K_2([(x_1,s), (x_2,s)]; [(y_1,t), (y_2,t)])}{\mc Z(x_2, s; y_2, t)} < \frac{\mc K_2([(x_1,s), (x_2,s)]; [(y_1,t), (y_3,t)])}{\mc Z(x_2, s; y_3,t)}
\end{align*}
and
\begin{align*}
\frac{\mc K_2([(x_1,s), (x_2,s)]; [(y_2,t), (y_3,t)])}{\mc Z(x_1, s; y_2, t)} < \frac{\mc K_2([(x_1,s), (x_2,s)]; [(y_1,t), (y_3,t)])}{\mc Z(x_1, s; y_1,t)}.
\end{align*}
\end{lemma} 

\begin{proof}
We give the proof for the first inequality as the second is analogous. Using \eqref{e.K definition}, expanding, and rearranging, we must show that
\begin{align*}
\MoveEqLeft[6]
\Bigl(\cZ(x_1,s;y_1,t)\cZ(x_2,s;y_2,t) - \cZ(x_1,s;y_2,t)\cZ(x_2,s;y_1,t)\Bigr)\cZ(x_2,s;y_3,t)\\
&< \Bigl(\cZ(x_1,s;y_1,t)\cZ(x_2,s;y_3,t) - \cZ(x_1,s;y_3,t)\cZ(x_2,s;y_1,t)\Bigr)\cZ(x_2,s;y_2,t).
\end{align*}
Cancelling and rearranging, this is equivalent to
\begin{align*}
\cZ(x_1,s;y_2,t)\cZ(x_3,s;y_3,t) - \cZ(x_1,s;y_3,t)\cZ(x_2,s;y_2,t) > 0.
\end{align*}
This is precisely \eqref{eq:FR-quad}.
\end{proof} 

\subsubsection{Continuous extension.}
Next we state the formal relation of $\cM_n$  and the multi-layer partition function $\cZ_n$ (recall \eqref{e.Z_k continuum}), through the following extension of $\cM_n$ to the boundary of $\Lambda_n\times \Lambda_n$. Though we do not use this fact directly anywhere in the paper, we state it for completeness and as it is used in the proofs of other facts we invoke (such as Lemma~\ref{l.K Z identity} ahead).

\begin{lemma}[\protect{\cite[Theorem 1.3]{LW}}]  \label{lem:ext} 
For any $s<t$ and $n\in\NN$, the function $(\bx, \by)\mapsto \cM_n(\bx,s;\by,t)$ (where $\mb x, \mb y\in\rDe_n$) almost surely extends continuously to $\Lambda_n\times \Lambda_n$, and the extension satisfies
\[
2^{-n(n-1)/2}(t-s)^{n(n-1)/2}\left(\prod_{i=1}^{n-1}i!\right)\cM_n(x\mathbf{1},s;y\mathbf{1},t)=\cZ_n(x,s;y,t),
\]
where $\mathbf{1}$ is the vector in $\R^n$ where each entry equals $1$. 
\end{lemma}

\begin{remark}
  As in Remark~\ref{r.Z_i vs mathcal Z_i}, we recall that $\cZ_2$ is not simply the continuum analog of $Z_2$, but rather is the partition function of two disjoint paths in the continuum.
\end{remark}

The following lemma relates $\cZ_2$ with $\cK_2$ and is a quick consequence of \cite[Proposition 3.2]{lun2016stochastic}. It will be used to move from Theorem~\ref{thm:disj-BK} to Theorem~\ref{thm:disj-BK-line-ensemble}, for which one needs to quantitatively approximate $\h_{t,2}$ (which is in turn defined in terms of $\cZ_2/\cZ_1$ due to \eqref{e.Z_1 Z_2 relation with h_1 h_2}) by $\mc K_2$.

\begin{lemma}\label{l.K Z identity}
The following holds almost surely. For any real numbers $s<t$, $a<b$, and $c<d$,
\begin{align}\label{e.Z K integral relation}
\log \left(1 + \frac{\mc K_2([(a,s), (b,s)]; [(c,t), (d,t)])}{\mc Z_1((a,s);(d,t))\mc Z_1((b,s);(c,t))}\right) = \int_a^b\int_c^d\frac{\mc Z_2((x,s);(y,t))}{\mc Z_1((x,s);(y,t))^2}\,\diff y\,\diff x.
\end{align}
\end{lemma}

One can thus move from $\cZ_2$ to $\cK_2$, for example, by take $a$ to be very close to $b$ and $c$ very close to $d$, on a scale say polynomial in $M^{-1}$ for a large parameter $M$. The resulting entropy contribution coming from the integral on the righthand side of \eqref{e.Z K integral relation} will result in a shift of $C\log M$. This is the source of the same shift in Theorem~\ref{thm:disj-BK-line-ensemble}, and the small probability of the event where $\mc Z_1$ or $\mc Z_2$ fluctuate too much on $[a,b]$ or $[c,d]$ produces the $\exp(-cM^2)$ error there, as we will see in Section~\ref{s.log gamma to CDRP}.

\begin{proof}[Proof of Lemma~\ref{l.K Z identity}]
\cite[Proposition 3.2]{lun2016stochastic} is a deterministic statement, which, in the $n=1$ case and applied to the process $\mc M$, yields
\begin{align*}
\log \left(\frac{\mc Z_1((a,s); (c,t))\mc Z_1((b,s);(d,t))}{\mc Z_1((a,s);(d,t))\mc Z_1((b,s);(c,t))}\right) = \int_a^b\int_c^d\frac{\mc Z_2((x,s);(y,t))}{\mc Z_1((x,s);(y,t))^2}\,\diff y\,\diff x;
\end{align*}
here we have implicitly used the identity relating $\mc M_2$ with $\mc Z_2$ (Lemma~\ref{lem:ext}). Rewriting the lefthand side using \eqref{e.K definition} completes the proof.
\end{proof}

\begin{remark}
  The lefthand side of \eqref{e.Z K integral relation} can be regarded as a positive temperature version of the ``difference profile'' \cite{BGH,ganguly2023local,GZ22,dauvergne2023last}, while the righthand side can be regarded as a positive temperature analogue of the shock measure \cite{bates2022hausdorff}. These two objects have been related at the zero temperature level in a number of works such as \cite{bates2022hausdorff}.
\end{remark}

\subsection{Line ensemble facts}

Here we collect a few statements on the KPZ line ensemble. First is a certain stochastic monotonicity property. The finite case is just \cite[Theorem 2.7]{GH22}, while the interval case is a quick consequence of the finite case that we prove in Appendix~\ref{s.misc proofs}.

\begin{lemma}[Monotonicity in conditioning]\label{l.conditional monotonicity}
Let $\cD\subset\R$ be an interval or a finite set of points. 
In the case that $\mc D$ is an interval, there exists a set $\Omega\subseteq \mc C(\mc D, \R)$ which has probability 1 under the law of a rate 2 Brownian motion on $\mc D$ with standard normal starting point, and in the case $\mc D$ is finite, we may take $\Omega = \R^{\mc D}$, such that the following holds.

If $f,g\in\Omega$ with $f(x)\geq g(x)$ for all $x\in \cD$, then the conditional law of $\h_{t}$ given $\h_{t,1}|_{\mc D} = f$ stochastically dominates the same given $\h_{t,1}|_{\mc D} = g$.
\end{lemma}

Next is a lower bound on the upper tail. 
We state it in terms of the scaled version $\hh_{t,1}$ and $\hh_{t,2}$ as defined in \eqref{e.scaling for KPZ}.  The proof in the case that $t$ is bounded away from zero was given in \cite{GH22}, and the case where $t$ is arbitrarily small was outlined in \cite[Appendix C]{ganguly2023brownian}. The proofs are reproduced in Appendix~\ref{s.misc proofs} for completeness.

\begin{proposition}[Lower bound on one-point upper tail]\label{p.lower bound on upper tail}
There exists $L_0$ such that, for all $t,L>0$ such that $L > (t^{-1/6}\vee 1)L_0$,
\begin{align*}
\P\left(\hh_{t,1}(0) > L\right) \geq \exp(-5L^{3/2}).
\end{align*}
\end{proposition}

\subsection{Regularity estimates}
Here we collect some estimates on the regularity of the curves $\h_{t,j}$ of the KPZ line ensemble as well as of the KPZ sheet. Since their proofs are immediate consequences of existing estimates, we defer them to Appendix~\ref{s.misc proofs}. We start by recalling the definition of the KPZ sheet.

Recall the solution to the stochastic heat equation \eqref{e.SHE definition} $(x,s;y,t)\mapsto \cZ(x,s;y,t)$. We define the \emph{KPZ sheet} to be the random continuous function $\h: \R^4_{\shortuparrow}\to \R$ given by
\begin{align*}
\h(x,s;y,t) := \log \cZ(x,s;y,t) + \frac{t-s}{12};
\end{align*}
the righthand side is almost surely positive for all $(x,s;y,t) \in \R^4_{\shortuparrow}$ simultaneously by \cite[Theorem~2.2]{AJRS}, so $\h$ is well-defined.

\begin{lemma}\label{l.two-point supremum for line ensemble}
There exist $C, c>0$ such that, for all $j\in\N$, $x,y\in\R$, $t>0$, and $K\geq 0$,
\begin{align*}
\P\left(\sup_{x,y\in[0,K], |x-y|\leq \varepsilon}|\hh_{t,j}(x) - \hh_{t,j}(y) + x^2-y^2| \geq M\varepsilon^{1/2}(\log \varepsilon^{-1})^{1/2}\right) \leq C(K+1)\exp(-cM^2).
\end{align*}
\end{lemma}

\begin{lemma}\label{l.two point for kpz sheet}
There exist $C,c>0$ such that for any $t>0$, $\varepsilon>0$, $K\geq 0$, and $M>0$,
\begin{align*}
\P\left(\sup_{\substack{|x_1|,|x_2|,|y_1|,|y_2| \leq K\\ \|(x_1,y_1) - (x_2,y_2)\| \leq \varepsilon}}|\h(x_1,0;y_1,t) - \h(x_2,0;y_2,t)| \geq M\varepsilon^{1/2}(\log\varepsilon^{-1})^{1/2}\right) \leq C(K+1)^2\exp(-cM^2).
\end{align*}

\end{lemma}

\section{From log-gamma polymers to the KPZ equation}\label{s.log gamma to CDRP}

In this section we use Theorem~\ref{thm:disj-BK-log-g-new} to give the proofs of Theorems~\ref{thm:disj-BK-line-ensemble}, \ref{thm:disj-BK-line-ensemble-technical}, and \ref{thm:disj-BK}. For Theorem~\ref{thm:disj-BK} we essentially just need to take the scaling limit of the log-gamma partition function to the CDRP partition function, for which we start by setting up the scalings in the next section. Then we prove Theorem~\ref{thm:disj-BK} in Section~\ref{s.cdrp bk proof} and Theorems~\ref{thm:disj-BK-line-ensemble} and \ref{thm:disj-BK-line-ensemble-technical} in Section~\ref{s.kpz line ensemble bk proof}.

\subsection{Scalings for log gamma} Recall from Section~\ref{s.log gamma} that $\theta$ is the parameter associated to the inverse-gamma distribution. For a scaling parameter $n$, we set $\theta = 2n^{1/2}$ and, for $\mb u, \mb v\in\Z^2$ which are ordered $\mb u\leq \mb v$ component-wise, let $Z^n_1(\mb u, \mb v)$ be the partition function as defined in \eqref{e.Z_i defintion} with this value of $\theta$. Recall $\R^4_{\shortuparrow}$ from Section~\ref{sss:pf}. Let $\mc Z^{n}$ be defined by
\begin{align}\label{e.Z^n definition}
\mc Z^{n}(x,s; y,t) = Z^n_1\left(ns+2n^{1/2}x, ns; nt+2n^{1/2}y, nt\right)\cdot n^{1/2}\cdot \left(2^{-1}(2n^{1/2}-1)\right)^{2n(t-s)+2n^{1/2}(y-x)}
\end{align}
for all $(x,s;y,t)\in\R^4_{\uparrow}$ such that the arguments of $Z^n_1$ are all integers; for all other $(x,s;y,t)\in\R^4_{\uparrow}$, we extend the definition by linear interpolation.

The following lemma was essentially proven in \cite{alberts2014intermediate}, and a similar statement was given in \cite[Proposition 5.10]{wu2019tightness}, so we will only outline its proof.

\begin{lemma}\label{l.Z joint convergence}
$\mc Z^n \to \mc Z$ weakly in the topology of uniform convergence on compact sets of $\R^4_{\uparrow}$ as $n\to\infty$.
\end{lemma}

\begin{proof}[Proof outline]
\cite[Theorem 2.7]{alberts2014intermediate} asserts the desired convergence for the full space-time field in a general polymer model where the vertex weight distributions have mean zero and variance one, which is not the case for the (rescaled) inverse gamma random variables we have. \cite[Proposition 5.10]{wu2019tightness} obtains the same statement for the marginal where $x=s=0$ in the case of rescaled inverse gamma random variables by observing that the argument of \cite{alberts2014intermediate} also holds if the variance converges to $1$ appropriately quickly (as $n\to\infty$). This observation does not rely on $x=s=0$ and so the same observation applied to the full convergence result of \cite[Theorem 2.7]{alberts2014intermediate} gives our result.
\end{proof}

\noindent Next we define the object which will converge to $\cK_2$ (recall its definition from \eqref{e.K definition}). First, let 
$$\R^6_{\shortuparrow, \leq} := \Bigl\{(x_1, x_2, s; y_1,y_2, t) : t >s \geq 0; (x_1,x_2), (y_1,y_2) \in \Lambda_2\Bigr\}.$$
Define  $\cK_2^n:\R^6_{\shortuparrow, \leq} \to \R$ by
\begin{equation}\label{e.K^n definition}
\begin{split}
  \MoveEqLeft[1]
\mc K^n_2\left([(x_1,s), (x_2,s)]; [(y_1,t), (y_2,t)]\right)\\
&= T\left([(ns+2n^{1/2}x_1, ns), (ns+2n^{1/2}x_2, ns)]; [(nt+2n^{1/2}y_1, nt), (nt+2n^{1/2}y_2, nt-1)]\right)\\
&\qquad\times n\times (2^{-1}(2n^{1/2}-1))^{4n(t-s)-1+2n^{1/2}(y_1+y_2-x_1-x_2)}
\end{split}
\end{equation}
for all $(x_1,x_2,y_1, y_2, s, t)$ such that $t\geq s+\frac{1}{n}$ and all the arguments of $T$ are integers, and for all other $(x_1,x_2,y_1, y_2, s, t)$ in the domain by linear interpolation. 

Here we use $nt-1$ in the second coordinate instead of $nt$ as in the first so as to match what appears in Corollary~\ref{c.bk coupling log gamma new}. It is because of this extra $-1$ term that the constant rescaling factor in \eqref{e.K^n definition} (i.e., the final line) contains an extra $-1$ in the exponent compared to the square of the analogous rescaling factor in \eqref{e.Z^n definition}, as can be seen by performing the substitutions $t\mapsto t-\frac{1}{n}$ and $y\mapsto y+\frac{1}{2n^{1/2}}$ in \eqref{e.Z^n definition}. An important point is that the fluctuations of the quantity in the first line of \eqref{e.K^n definition} after replacing $nt-1$ by $nt$ differs by a $1+o(1)$ factor, modulo the just mentioned deterministic modification in the scaling. Thus, as we see next, in the limit the distinction is not present.

In the following we regard $\cK_2^n$ as a continuous function defined on all of $\R^6_{\shortuparrow, \leq}$ by embedding the above definition and extending to the remainder of the domain in an arbitrary manner that respects continuity.

\begin{lemma}\label{l.K convergence}
As $n\to\infty$, in the topology of uniform convergence on compact sets of $\R^4_{\shortuparrow}\times \R^6_{\shortuparrow, \leq}$,
$\smash{(\cZ^n,\mc K^n_2) \stackrel{d}{\to} (\cZ, \mc K_2)}.$
\end{lemma}

\begin{proof}
Since $\cK_2$ and $\cK^n_2$ are the same continuous function of $\cZ$ and $\cZ^n$ by \eqref{e.K definition} and Lemma~\ref{lem:disj-alter-def}, respectively, it suffices to prove that both the marginals of $(\cZ^n,\mc K^n_2)$ converge to the correct limits. That $\smash{\cZ^n\stackrel{d}{\to} \cZ}$ is Lemma~\ref{l.Z joint convergence}. So we turn to the convergence of $\cK^n_2$.

First, by Lemma~\ref{lem:disj-alter-def}, we can express $\cK_2^n$ as a determinant of a matrix whose entries are given by $Z^n_1$ evaluated at the appropriate points. By the joint convergence of $Z^n_1$ after rescaling to $\cZ$ as given in Lemma~\ref{l.Z joint convergence}, we obtain that the limit as $n\to\infty$ of this determinant is exactly the one appearing in the definition \eqref{e.K definition} of $\cK_2$ as a determinant of $\cZ$ (note that the required rescaling for the factors of $Z_1(ns+2n^{1/2}x_i,ns; nt+2n^{1/2}y, nt-1)$ matches what is present by the discussion after \eqref{e.K^n definition}). This completes the proof.
\end{proof}

\subsection{Proof of the limiting BK inequality for disjoint polymers in the CDRP}\label{s.cdrp bk proof}

Here we give the proof of Theorem~\ref{thm:disj-BK}. We start with two basic but useful facts about weak convergence.

\begin{lemma}\label{l.weak limit of conditioned measures}
Let $S$ be a metric space with Borel $\sigma$-algebra $\mc S$. Let $X_1, X_2, \ldots$ and $X$ be random elements taking values in $S$ with $X_n \smash{\stackrel{d}{\to}} X$. Let $\msf E\in \mc S$ be a continuity set for $X$, i.e., $\P(X\in \partial\msf E) = 0$, where $\partial \msf E = \bar{\msf E}\setminus \msf E^{\circ}$ is its topological boundary, and suppose $\P(X\in\msf E) > 0$. Then $X_n$ conditioned on $X_n\in\msf E$ converges in distribution to $X$ conditioned on $X\in\msf E$.
\end{lemma}

\begin{proof}
It suffices to show that, for any bounded continuous function $f:S\to\R$, as $n\to\infty$,
$
\E[f(X_n)\mid X_n\in\msf E] \to \E[f(X)\mid X\in\msf E].
$
Write $\E[f(X_n)\mid X_n\in\msf E] = \E[f(X_n)\one_{X_n\in\msf E}]/\P(X_n\in\msf E)$. Then what we must show follows from the Portmanteau and continuous mapping theorems since $\msf E$ is a continuity set for $X$. First, the Portmanteau theorem implies $\P(X_n\in\msf E)\to \P(X\in\msf E)$, and, since $x\mapsto f(x)\one_{x\in\msf E}$ is almost surely continuous at $X$, the continuous mapping theorem implies $\E[f(X_n)\one_{X_n\in\msf E}]\to \E[f(X)\one_{X\in\msf E}]$.
\end{proof}

\begin{proof}[Proof of \Cref{thm:disj-BK}]
Recall we wish to bound the conditional probability 
\begin{equation}\label{e.conditional prob to bound}
\PP\left(\log \cK_2((x_1,x_2), s; (y_1, \bm\cdot), t)|_{[y_1,y_1+K]} - \log \cZ(x_1, s; y_1, t)  \in \msf A  \mid \log \cZ(x_1, s; \bm\cdot, t)|_{[y_1-R, y_1]}\right).
\end{equation}
(The case where $+K$ and $-R$ are simultaneously replaced by $-K$ and $+R$ is immediate after recalling the distributional symmetry of the processes under reflection around the origin, Lemma~\ref{l.Z symmetries}.) By spatial and temporal translation invariance (Lemma~\ref{l.Z symmetries}) of $\cZ$ and $\cK_2$, it suffices to prove the case $x_1=s=0$. By shear invariance of $\cK_2$ and $\mc Z$ (again Lemma~\ref{l.Z symmetries}), it further suffices to take $y_1=0$.

Let $\{\mc U_m\}$ be an increasing sequence of finite sets, i.e., $\mc U_1 \subset \mc U_2 \subset \ldots$ such that $\cup_{m=1}^\infty \mc U_m$ is dense in $[-R, 0]$. Note that if we consider the conditional probability above in \eqref{e.conditional prob to bound} but conditioned on $\log \cZ(0, 0; \bm\cdot, t)|_{\mc U_n}$ rather than $\log \cZ(0, 0; \bm\cdot, t)|_{[-R, 0]}$, then this sequence of conditional probabilities forms a martingale (in $n$) whose limit is \eqref{e.conditional prob to bound} (due to the denseness assumption on $\{\mc U_n\}$ and the continuity of $\mc Z$). So it suffices to prove the desired bound on the conditional probability conditioned on the values of $\log \cZ(0, 0; \bm\cdot, t)$ on a finite set; this is useful as we can avoid technicalities regarding the convergence of probabilities conditioned on equaling a given element in the space of continuous functions and instead work in finite-dimensional settings.

Now we turn to prove the desired bound when conditioned on the values of $\log \cZ(x_1, s; \bm\cdot, t)$ on a finite set $\mc U \subset [-R, 0]$. Let $g : \mc U \to \R$.   For any $\varepsilon\in (0,\infty]$ fixed, let $\mu_n^\varepsilon$ be the law of %
\begin{align*}
\MoveEqLeft[16]
\frac{\mc K^n_2([(0,0), (x_2, 0)]; [(0,t), (\bm\cdot,t-\frac{1}{n})])|_{[0,K]}}{\mc Z^n(0,0; 0,t)}\\
&\text{conditioned on}\quad \mc Z^n(0,0; \bm\cdot,t)|_{\mc U} \in \prod_{x\in\mc U}[g(x), g(x)+\varepsilon]%
\end{align*}
and $\nu_n$ be the law of $\mc Z^n(x_2,0; \bm\cdot,t-\tfrac{1}{n})$. Let $\mu^\varepsilon$ be the law of 
\begin{align*}
\frac{\mc K_2([(0,0), (x_2,0)]; [(0,t), (\bm\cdot,t)])|_{[0,K]}}{\mc Z(0,0; 0,t)} \quad\text{conditioned on}\quad \mc Z(0,0; \bm\cdot,t)|_{\mc U} \in \prod_{x\in\mc U}[g(x), g(x)+\varepsilon],
\end{align*}
and $\nu$ be the law of $\mc Z(x_2,0; \bm\cdot,t)$. 

Next we wish to take the weak limit as $n\to\infty$ and make use of Lemma~\ref{l.weak limit of conditioned measures}, whose assumptions we now verify. First, by Lemma~\ref{l.K convergence} we know that $\smash{(\cZ^n, \cK^n_2) \stackrel{d}{\to} (\cZ,\cK_2)}$. Second, $\cZ$ and $\cK_2$ are atomless (the latter following from the former by the definition \eqref{e.K definition}) and $\cZ(0,0; \bm\cdot, t)|_{\mc U}$ has full support, which follows from the Brownian Gibbs property of the KPZ line ensemble (see Section~\ref{ss:legp} or \cite{CH14}) or by \cite[Theorem 1.2]{chen2021regularity}. So Lemma~\ref{l.weak limit of conditioned measures} yields that, as $n\to\infty$,
\begin{align}\label{e.conditioned measure weak limit}
\mu^\varepsilon_n \stackrel{d}{\to} \mu^\varepsilon \quad\text{and}\quad \nu_n \stackrel{d}{\to} \nu.
\end{align}

Now we wish to apply Corollary~\ref{c.bk coupling log gamma new}. We make substitutions for the coordinates and quantities appearing in that statement so as to put things in the correct scaling. First, replace $a$ by $2n^{1/2}x_2$, $b$ by $nt$, and $b'$ by $nt+2n^{1/2}y_2$. Next, we replace $x\mapsto f(x)$ by $x\mapsto g(x)n^{1/2}(2^{-1}(2n^{1/2}-1))^{2nt}$. Now, Corollary~\ref{c.bk coupling log gamma new} and the definition \eqref{e.K^n definition} of $\mc K^n_2$ yield that there exists a coupling of $X^\varepsilon_n(\bm\cdot)\sim\mu^\varepsilon_n$ and $Y_n(\bm\cdot)\sim \nu_n$ such that, almost surely, for all $x\in[0, K]$,
\begin{align*}
X^\varepsilon_n(x) \leq Y_n(x).
\end{align*}
Combining \eqref{e.conditioned measure weak limit} with the previous display, the Skorohod representation theorem yields a coupling of $X^\varepsilon(\bm\cdot)\sim\mu^\varepsilon$  and $Y(\bm\cdot)\sim\nu$ such that $X^\varepsilon(x) \leq Y(x)$ for all $x\in[0, K]$.%

Setting $g(x)=e^{f(x)}$ for all $x\in \mc U$, the coupling yields, with $\varepsilon'_x = \log(1+\varepsilon e^{-f(x)})$ for all $x\in\mc U$ and for any increasing Borel measurable set $\msf A$, that
\begin{align*}
\MoveEqLeft[23]
\PP \Biggl(\hspace{-0.8cm}\parbox[c]{8.1cm}{\centering$\bigl(\log \cK([(0,0)(x_2, 0)]; [(0,t), (\bm\cdot, t)])|_{[0, K]}$\\[4pt]
\hspace{4.2cm}$- \log \cZ(0, 0; 0, t)\bigr) \in \msf A$}\   \middd \log \cZ(0, 0; \bm\cdot, t)|_{\mc U} \in \prod_{x\in\mc U}[f(x), f(x)+\varepsilon'_x] \Biggr) \\ 
&\hspace{2cm}\leq \PP \Bigl( \log \cZ(x_2, 0; \bm\cdot, t)|_{[0, K]} \in \msf A \Bigr). 
\end{align*}
Taking $\varepsilon\to 0$ and using the Portmanteau theorem completes the proof in the case of conditioning on $\log\cZ(0, 0; \bm\cdot, t) = f$.  Since the righthand side has no dependence on $f$, averaging the inequality over the conditioning variable yields the same inequality when conditioned on $\log \cZ(0, 0; \bm\cdot, t) \geq f$.
\end{proof}

\subsection{Proof of BK inequality for the KPZ line ensemble}\label{s.kpz line ensemble bk proof}

Here we give the proofs of Theorems~\ref{thm:disj-BK-line-ensemble} and \ref{thm:disj-BK-line-ensemble-technical}. We will need a way to move between $\h_{t,2}$ (equivalently, $\mc Z_2$) and $\mc K_2$, since we intend to use the analogous BK inequality statement we have already proven for $\cK_2$ (Theorem~\ref{thm:disj-BK}). This is provided by the next lemma.

\begin{lemma}\label{l.Z K comparison}
Let $K>0$. There exist $C,c>0$ such that, for $\varepsilon>0$, $s<t$, $x\in\R$, and $1\leq \lambda \leq \varepsilon^{-1/2}(\log\varepsilon^{-1})^{-1/2}$, the following holds. With probability at least $1-3K\exp(-c\lambda^2)$, for all $|y|\leq K$,
\begin{align*}
\mc Z_2\left((x,s); (y,t)\right) \leq \varepsilon^{-2}(1+C\lambda\varepsilon^{1/2}(\log \varepsilon^{-1})^{1/2})\cdot \mc K_2\left([(x,s), (x+\varepsilon,s)]; [(y,t), (y+\varepsilon,t)]\right).
\end{align*}
For any fixed $y\in\R$, the same holds with probability at least $1-C\exp(-c\lambda^2)$.
\end{lemma}

Here, as we will see, the $\varepsilon^{-2}$ comes from evaluating the integral in Lemma~\ref{l.K Z identity} over a pair of $\varepsilon$-sized neighborhoods, and the $\varepsilon^{1/2}(\log\varepsilon^{-1})^{1/2}$ term and error probability bound come from Brownian modulus of continuity estimates for $\cZ_1$ and $\cZ_2$.

\begin{proof}[Proof of Lemma~\ref{l.Z K comparison}]
By spatial and temporal translation invariance we may assume $x=s=0$. Now by \eqref{e.Z_1 Z_2 relation with h_1 h_2}, for any $w\in\R$,
$$\log \mc Z_2\left((0,0); (w,t)\right) = \h_{t,1}(w) + \h_{t,2}(w) - t/6.$$
Then Lemma~\ref{l.two-point supremum for line ensemble} applied to $\h_{t,1}$ and $\h_{t,2}$ separately, along with the triangle inequality,  yields that, on an event with probability at least $1-K\exp(-c\lambda^2)$, for all $|y|\leq K$ and all $w\in [y,y+\varepsilon]$, $|\log \mc Z_2\left((0,0); (w,t)\right) - \log \mc Z_2\left((0,0); (y,t)\right)| \leq \lambda\varepsilon^{1/2}(\log\varepsilon^{-1})^{1/2},$ which implies
\begin{align}\label{e.Z_2 continuity}
\mc Z_2\left((0,0); (w,t)\right) \leq \mc Z_2\left((0,0); (y,t)\right)(1+CR\varepsilon^{1/2}(\log\varepsilon^{-1})^{1/2})
\end{align}
for all such $w$. Similarly we also have that on an event of probability at least $1-K\exp(-c\lambda^2)$, for all $|y|\leq K$ and $w\in[y,y+\varepsilon]$,
\begin{align}\label{e.Z_1 continuity}
\frac{\mc Z_1\left((0,0); (w,t)\right)}{\mc Z_1\left((0,0); (y,t)\right)} \in [1-C\lambda\varepsilon^{1/2}(\log\varepsilon^{-1})^{1/2}, 1+C\lambda\varepsilon^{1/2}(\log\varepsilon^{-1})^{1/2}],
\end{align}
and, by Lemma~\ref{l.two point for kpz sheet}, the same also holds with $(0,0)$ in the numerator replaced by $(\varepsilon,0)$.

Now we apply Lemma~\ref{l.K Z identity} with $a=0$, $b=\varepsilon$, $c=y$, $d=y+\varepsilon$ and use $\log(1+w) < w$ for $w>0$ to obtain that, deterministically,
\begin{align*}
\int_{0}^\varepsilon\int_{y}^{y+\varepsilon}\frac{\mc Z_2((w,0);(z,t))}{\mc Z_1((w,0);(z,t))^2}\,\diff z\, \diff w \leq \frac{\mc K_2([(0,0), (\varepsilon,0)]; [(y,t), (y+\varepsilon,t)])}{\mc Z_1((0,0);(y+\varepsilon,t))\mc Z_1((\varepsilon,0);(y,t))}.
\end{align*}
Applying the estimates \eqref{e.Z_2 continuity} and \eqref{e.Z_1 continuity} to the previous display yields that, on an event of probability at least $1-3K\exp(-c\lambda^2)$, for all $|y|\leq K$,
\begin{align*}
\MoveEqLeft[10]
\varepsilon^2\frac{\mc Z_2((0,0);(y,t))}{\mc Z_1((0,0);(y,t))^2}\cdot\frac{(1-C\lambda\varepsilon^{1/2}(\log\varepsilon^{-1})^{1/2})}{(1+C\lambda\varepsilon^{1/2}(\log\varepsilon^{-1})^{1/2})}\\
&\leq \frac{\mc K_2([(0,0), (\varepsilon,0)]; [(y,t), (y+\varepsilon,t)])}{\mc Z_1((0,0);(y,t))^2}\cdot(1+C\lambda\varepsilon^{1/2}(\log\varepsilon^{-1})^{1/2}).
\end{align*}
Rearranging, cancelling common terms, and relabeling $C$ completes the proof in the case of controlling simultaneously over all $y\in[0,K]$. In the case of a fixed $y\in\R$, we use the stationarity of $w\mapsto \cZ_2( (0,0), (w,t))$ and Lemma~\ref{l.two-point supremum for line ensemble} with $K=1$ to obtain \eqref{e.Z_2 continuity} and \eqref{e.Z_1 continuity} at $w=y$ with probability at least $1-C\exp(-c\lambda^2)$. This completes the proof.
\end{proof}

Now we may turn to giving the proofs of Theorems~\ref{thm:disj-BK-line-ensemble} and \ref{thm:disj-BK-line-ensemble-technical}. For the convenience of the reader, we restate Theorem~\ref{thm:disj-BK-line-ensemble-technical} below.

\BKlineensembletechnical*

The shift by $Ct^{-1/3}\log M$ in \eqref{e.general bk} will arise in the proof from an invocation of Lemma~\ref{l.Z K comparison} (with the parameters set as functions of $M$) to move from $\cZ_2$ to $\cK_2$, and this is also the source of the $\exp(-cM^2)$ in the error term.

\begin{remark}\label{r.other choices of error term}
Observe that, by taking $M= M't$, for $M'$ a constant depending on $f$, Theorem~\ref{thm:disj-BK-line-ensemble-technical} yields
\begin{align*}
\MoveEqLeft[18]
\P\Bigl(\hh_{t,2}|_{[y,y+K]} - Ct^{-1/3}\log(M't) \in \msf A \midd \hh_{t,1}|_{[y-R, y]} = f \Bigr)\\
&\leq \P\left(\hh_{t,1}|_{[y,y+K]} \in \msf A\right) + CKt^{2/3}\exp(-c(M')^2t^2).
\end{align*}
In particular, as $t\to\infty$, one obtains the zero temperature version of the same inequality, i.e., for the parabolic Airy line ensemble \cite[Theorem 2.7]{GH22}, using the convergence of the KPZ line ensemble to the parabolic Airy line ensemble (which follows from combining the main results of \cite{Wu21} and \cite{AHALE}). This is also an immediate consequence of the full BK inequality as proven in \cite{GH22}. 
\end{remark}

\begin{proof}[Proof of Theorem~\ref{thm:disj-BK-line-ensemble} (assuming Theorem~\ref{thm:disj-BK-line-ensemble-technical})]
This follows immediately after taking $R=0$ and $f(y) = L$ in Theorem~\ref{thm:disj-BK-line-ensemble-technical}, and recalling that $\P(\hh_{t,1}(y) \geq L)  = \P(\hh_{t,1}(0) \geq L + y^2) \geq \exp(-c(L+y^2)^{3/2})$ if $L > L_0t^{-1/6}$ from Proposition~\ref{p.lower bound on upper tail}. Since $M>C(L+y^2)^{3/4}$ for a large enough constant $C$ is assumed, it follows that
\begin{align*}
\frac{C(K+1)t^{2/3}\exp(-cM^2)}{\P(\hh_{t,1}(y) \geq L)} \leq C(K+1)t^{2/3}\exp(-cM^2)
\end{align*}
(indeed, as mentioned after Theorem~\ref{thm:disj-BK-line-ensemble}, this lower bound on $M$ was assumed precisely for the above inequality to hold). This completes the proof.
\end{proof}

\begin{proof}[Proof of Theorem~\ref{thm:disj-BK-line-ensemble-technical}]
First, by stationarity of $w\mapsto \hh_{t,i}(w)+w^2$ jointly across $i$ (see, e.g., \cite[Proposition 1.3]{nica2021intermediate}), we may assume without loss of generality that $y=0$ (without modifying the values of $C$ and $c$, as can be seen by inspecting \eqref{e.general bk}). 

Next, it suffices to prove \eqref{e.general bk} in the case of $y+\bm\cdot$ rather than $y-\bm\cdot$. This is because (since we have reduced to $y=0$) $w\mapsto \mc Z_i\left((0,0);(w,t)\right)$ has the same distribution as $w\mapsto \mc Z_i\left((0,0);(-w,t)\right)$ for $i=1$ and $2$ jointly (which follows immediately from the definition in terms of the chaos expansions and the analogous symmetry of the underlying white noise).

Now, by monotonicity in conditioning (Lemma~\ref{l.conditional monotonicity}),
\begin{align*}
\MoveEqLeft[16]
\P\Bigl(\hh_{t,2}(\bm \cdot)|_{[0,K]} - Ct^{-1/3}\log M \in \msf A \midd \hh_{t,1}|_{[-R, 0]} = f \Bigr)\\
&\leq \P\Bigl(\hh_{t,2}(\bm \cdot)|_{[0,K]} - Ct^{-1/3}\log M \in \msf A \midd \hh_{t,1}|_{[-R, 0]} \geq f \Bigr),
\end{align*}
so it suffices to bound the latter quantity. We assume $f:[y-R,y]\to\R\cup\{-\infty\}$ is upper semicontinuous. In the remaining argument we will work under the conditioning $\hh_{t,1}|_{[-R, 0]} \geq f$.

Now, from \eqref{e.Z_1 Z_2 relation with h_1 h_2} and \eqref{e.scaling for KPZ} we may write $\hh_{t,2}(z)$ as
\begin{align*}
\hh_{t,2}(z) = t^{-1/3}\log \frac{\mc Z_2\left((0,0); (zt^{2/3},t)\right)}{\mc Z\left((0,0); (zt^{2/3},t)\right)} + \frac{t^{2/3}}{12}.
\end{align*}

We apply Lemma~\ref{l.Z K comparison} with $\varepsilon = M^{-4}$, $\lambda = M = \varepsilon^{-1/4} \ll \varepsilon^{-1/2}(\log\varepsilon^{-1})^{-1/2}$, $x=0$, and $s=0$, which yields that, with probability at least $1-3(K+1)t^{2/3}\exp(-cM^2)$, for all $|z|\leq K$, 
\begin{align}
\log \mc Z_2((0,0); (zt^{2/3},t))
&\leq  C + 8\log M + \log \mc K_2([(0, 0), (M^{-4}, 0)]; [(zt^{2/3}, t), (zt^{2/3}+M^{-4}, t)])\nonumber\\
&\leq  C\log M + \log \mc K_2([(0, 0), (M^{-4}, 0)]; [(zt^{2/3}, t), (zt^{2/3}+M^{-4}, t)]). \label{e.Z K final inequality}
\end{align}
We denote the event that the previous display occurs by $\mc E$, so that 
\begin{equation}\label{e.E^c prob bound}
\P(\mc E^c)\leq 3(K+1)t^{2/3}\exp(-cM^2).
\end{equation}

We have now related $\cZ_2$ to $\cK_2$. Ultimately we will need to invoke Theorem~\ref{thm:disj-BK}. But observe that, taking $x_1=0$, the probability appearing in that statement conditions on $\cZ(0,0; \bm\cdot, t)|_{[y_1-R,y_1]}$ and considers $\cK_2([(0,0), (x_2, 0)]; [(y_1, t), (\bm\cdot, t)])|_{[y_1,y_1+K]}$, i.e., the $x$-coordinate of the first ending point in $\cK_2$ is fixed and matches the right endpoint of the interval on which $\cZ$ is being conditioned. However, note that the $x$-coordinate of the first endpoint in the $\cK_2$ term in \eqref{e.Z K final inequality} is not fixed.

To address this, we next apply Lemma~\ref{l.K Z ordered inequality} with $x_1 = 0$, $x_2 = M^{-4}$, $y_1=0$, $y_2 = zt^{2/3}$, $y_3 = zt^{2/3}+M^{-4}$ to obtain that, almost surely, for all $z\in[0,K]$,
\begin{align*}
\frac{\mc K_2([(0, 0), (M^{-4}, 0)]; [(zt^{2/3}, t), (zt^{2/3}+M^{-4}, t)])}{\mc Z\left((0,0); (zt^{2/3},t)\right)} \leq \frac{\mc K_2([(0, 0), (M^{-4}, 0)]; [(0, t), (zt^{2/3}+M^{-4}, t)])}{\mc Z\left((0,0); (0,t)\right)},
\end{align*}
so that, indeed, we may work with a quantity in which the $x$-coordinate of the first endpoint in $\cK_2$ is fixed to be zero (which is also the right endpoint of the interval that $\hh_{t,1}$ will be  conditioned on).
Thus we see that, on $\mc E$, for all $z\in[0,K]$,
\begin{align}\label{e.h and K inequality}
\hh_{t,2}(z) \leq Ct^{-1/3}\log M + t^{-1/3}\log \frac{\mc K_2([(0, 0), (M^{-4}, 0)]; [(0, t), (zt^{2/3}+M^{-4}, t)])}{\mc Z\left((0,0); (0,t)\right)} + \frac{t^{2/3}}{12}.
\end{align}

For notational convenience, for all $w\in\R$, let 
\begin{equation}\label{e.tilde K}
\tilde{\mc K}^{(t)}_2(w) := \mc K_2([(0, 0), (M^{-4}, 0)]; [(0, t), (wt^{2/3}+M^{-4}, t)]).
\end{equation}
Next note that since $\msf A$ is increasing, $f\in \msf A$ and $g\geq f$ implies $g\in\msf A$. With this and \eqref{e.h and K inequality}, we obtain
\begin{align}
\MoveEqLeft[1]
\P\Bigl(\hh_{t,2}(\bm \cdot)|_{[0,K]} - Ct^{-1/3}\log M \in \msf A \midd \hh_{t,1}|_{[-R, 0]} \geq f \Bigr)\nonumber\\
&\leq \P\left(t^{-1/3}\log \frac{\tilde{\mc K}^{(t)}_2(\bm\cdot)|_{[0,K]}}{\mc Z\left((0,0); (0,t)\right)} + \tfrac{1}{12}t^{2/3} \in \msf A \middd \hh_{t,1}|_{[-R, 0]} \geq f\right) + \P\left(\mc E^c \midd \hh_{t,1}|_{[-R, 0]} \geq f\right). \label{e.BK derivation final inequality}
\end{align}

We wish to now invoke Theorem~\ref{thm:disj-BK} to bound the first term in \eqref{e.BK derivation final inequality}. We set $x_1 = 0$, $x_2 = M^{-4}$, $y_1 = 0$, and $s=0$. Then Theorem~\ref{thm:disj-BK} (recalling \eqref{e.tilde K} and that $\hh_{t,1}(0)$ is an increasing linear function of $\cZ( (0,0); (0,t))$) implies
\begin{align*}
\MoveEqLeft[12]
\P\left(t^{-1/3}\log \frac{\tilde{\mc K}^{(t)}_2(\bm\cdot)|_{[0,K]}}{\mc Z\left((0,0); (0,t)\right)} + \tfrac{1}{12}t^{2/3} \in \msf A \midd \hh_{t,1}|_{[-R, 0]} \geq f\right)\\
&\leq \P\left(z\mapsto t^{-1/3}\log \mc Z(M^{-4},0; zt^{2/3}+M^{-4},t)|_{[0,K]} + \tfrac{1}{12}t^{2/3} \in \msf A\right).
\end{align*}
By translation invariance, the final probability equals 
$$\P\left(z\mapsto t^{-1/3}\log \mc Z(0,0; zt^{2/3},t)|_{[0,K]} + \tfrac{1}{12}t^{2/3}\in \msf A\right) = \P\left(\hh_{t,1}(\bm\cdot)|_{[0,K]}\in \msf A\right)$$
by \eqref{e.Z_1 Z_2 relation with h_1 h_2}, which is our final upper bound on the first term of \eqref{e.BK derivation final inequality}. By the trivial bound, the second term of \eqref{e.BK derivation final inequality} is upper bounded by
\begin{align*}
\frac{\P\left(\mc E^c\right)}{\P\left(\hh_{t,1}|_{[-R, 0]} \geq f\right)}.
\end{align*}
As already noted in \eqref{e.E^c prob bound}, the numerator is upper bounded by $3(K+1)t^{2/3}\exp(-cM^2)$. %
This completes the proof.
\end{proof}

\section{Generalizations}    \label{s.generalizations}

In this section we discuss how our method actually extends to other multi-point partition functions. In Section~\ref{s.generalization log gamma}, we give a generalization in the log-gamma model, Theorem~\ref{thm:disj-BK-log-g-ext}, that follows easily from the arguments presented in Section~\ref{s.bk for log gamma}. In Sections~\ref{s.generalization.cdrp} and \ref{s.generalization.kpz} we indicate how Theorem~\ref{thm:disj-BK-log-g-ext} would lead to results for the CDPR and KPZ line ensemble, respectively, modulo some technical ingredients which are currently not available in the literature. 
To keep the exposition brief, instead of providing all the details, we simply indicate in words how to adapt the the proofs from the previous sections.

\subsection{More points in log-gamma}\label{s.generalization log gamma} To begin with, for the log-gamma polymer model, the proof of \Cref{thm:disj-BK-log-g-new} also leads to the following inequality (here we follow the notation of \Cref{s.log gamma}).

Let $\theta>0$. For any integers $m, n, w$ satisfying $m, n > w \ge 1$, let $\Omega_{m,n,w}$ be the following set of integer tuples:
\begin{align*}
\Omega_{m,n,w} := \left\{ (a_1, \ldots, a_k; \ell_1, \ldots, \ell_k)\in\N^{2k} : \parbox[c]{5.5cm}{\centering $k\in\N, w<a_1<\cdots <a_k\le m$,\\ $n-w>\ell_1>\cdots>\ell_k \ge 0$} \right\}.
\end{align*}
Here, $w$ will be the number of ``padding'' points that will form the leftmost starting point and topmost ending points in our theorem statement, and the conditions above ensure that there is sufficient space for these points.

Denote $(1,1)^k=[(1,1), \ldots, (k,1)]$ and $(m, n)^k =[(m,n), \ldots, (m,n-k+1)]$ for each $k\in \N$. Also, for more compact notation, for a vector $\mb a = (a_1, \ldots, a_k)$, let $(\mb a, 1) = ((a_1,1), \ldots, (a_k,1))$, and for a vector $\bm \ell = (\ell_1, \ldots, \ell_k)$, let $(m,\bm\ell) = ( (m,\ell_1), \ldots, (m,\ell_k) )$. Finally, recall the definitions of $J[m,n]$ from \eqref{e.J definition} and of $Z^{m,n}_j(i)$ from \eqref{e.Z^m,n first} and \eqref{e.Z^m,n second}.

\begin{theorem}   \label{thm:disj-BK-log-g-ext}
 For any increasing Borel measurable set $\msf A \subseteq \R^{|\Omega_{m,n,w}|}$, almost surely
\begin{equation*}
\begin{split}
\MoveEqLeft[23]
  \PP \Bigg( \Bigg\{\frac{T( [(1,1)^w, (\mb a, 1)], [(m,n)^w, (m,\bm\ell)] )}{ T_w((1,1), (m,n))}\Bigg\}_{ (\mb a; \bm\ell) \in \Omega_{m,n,w} } \!\!\! \in \msf A 
  \middd \Bigl\{Z^{m,n}_j(i) : (i,j) \in J[m,n], 1\leq j\leq w \Bigr\}\Bigg) \\
  &\le \PP \left( \Bigl\{T( [(\mb a,1)], [(m,\bm \ell)] )\Bigr\}_{ (\mb a; \bm \ell) \in \Omega_{m,n,w} } \!\! \in \msf A \right).    
\end{split}
\end{equation*}
\end{theorem}

Its proof is essentially verbatim that of \Cref{thm:disj-BK-log-g-new}. Indeed, first, Proposition~\ref{p.extended invariance} expresses
\[T( [(a_1,1), \ldots, (a_k,1)], [(m,\ell_1), \ldots, (m,\ell_k)] )\]
as a $k$-path partition function in the line ensemble, i.e., in terms of $\{Z^{m,n}_j(i)\}_{(i,j)\in J[m,n]}$.
Then, \Cref{corr:cross-line} can be generalized to express the ratio 
\[ \frac{T( [(1,1)^w, (\mb a, 1)], [(m,n)^w, (m,\bm\ell)] )}{ T_w((1,1), (m,n))} = \frac{T( [(1,1)^w, (a_1,1), \ldots, (a_k,1)], [(m,n)^w, (m,\ell_1), \ldots, (m,\ell_k)] )}{ T_w((1,1), (m,n))},\]
also in terms of  $\{Z^{m,n}_j(i)\}_{(i,j)\in J[m,n]}$; more precisely, as a $k$-path partition function in the line ensemble obtained by excluding the top $w$ lines of $Z^{m,n}$. The idea of the proof is exactly the same as that of Corollary~\ref{corr:cross-line}, namely the $w$ paths from $(1,1)^{w,\shortuparrow}$ to $(m,n)^{w,\shortrightarrow}$ in the line ensemble are frozen and fully occupy the top $w$ lines, and the contribution of these paths to the numerator of the previous display is exactly cancelled by the denominator.

Second, \Cref{lem:shift} would imply a variant of \Cref{cor:reweight}, where the law of $\{Z^{m,n}_{j+w}(i+w)\}_{(i,j)\in J[m-w,n-w]}$ conditional on $\{Z^{m,n}_j(\bm\cdot)\}_{i=1}^w$ is given by the (unconditioned) law of $$\{Z^{m-w,n-w}_{j}(i)\}_{(i,j)\in J[m-w,n-w]}$$ with a certain reweighting factor.
The reweighting is of a form similar to $\Gamma_g$ (from \eqref{e.Gamma_g definition}), and is also negatively associated with any increasing event, by the proof of \Cref{lem:neg-cor-fA}. 
Then the same arguments in the proof of \Cref{thm:disj-BK-log-g-new} lead to \Cref{thm:disj-BK-log-g-ext}.

\subsection{CDRP limit}\label{s.generalization.cdrp}
If one passes \Cref{thm:disj-BK-log-g-ext} through the log-gamma to the CDRP scaling limit, as in our proof of \Cref{thm:disj-BK}, one would get a generalization of \Cref{thm:disj-BK}. However, such a scaling limit result for the multi-path partition function with general endpoints does not seem to be present in the literature, which we will expand on a little more shortly. For this reason, the following simply indicates the form of the result one would obtain for the CDRP from Theorem~\ref{thm:disj-BK-log-g-ext} if such a scaling limit was established. 

Recall the notation set up in \Cref{s.prelim tools}.
Fix any real numbers $x, y\in\R$, and $w, n\in \N$. Let 
$\msf A\subseteq \prod_{i=1}^n\mc C(\Lambda_i([x,x+K])\times\Lambda_i([y,y+K]),\R)$
 for some $K>0$ be an increasing Borel measurable set, where by increasing we mean that if $\mb f = (f_1, \ldots, f_n) \in \msf A$ and $\mb g = (g_1, \ldots, g_n)$ is continuous, has the same domain as $\mb f$, and satisfies $g_i(z) \geq f_i(z)$ for each $i\in\intint{1,n}$ and $z$ in the domain, then $\mb g \in \msf A$. Let $\mathbf{1}\in\R^w$ be the vector whose entries are all $1$. Then for any $s<t$, $R>0$, almost surely 
\begin{equation}   \label{eq:multi-disj-BK}
\begin{split}
\MoveEqLeft[16]
\PP \Biggl( \left\{\log \frac{\cM_{w+i}((x\mathbf{1}, \bm\cdot), s; (y\mathbf{1}, \bm\cdot), t)|_{\Lambda_i([x, x+K]) \times \Lambda_i([y, y+K]) }}{\cM_w((x\mathbf{1}), s; (y\mathbf{1}), t)}  \right\}_{i=1}^n  \in \msf A
 \ \middd\  \parbox[c]{4.7cm}{$\{\log \cZ_j(x, s; \bm\cdot, t)\}_{j=1}^w|_{[y-R, y] },$\\ $\;\{\log \cZ_j(\bm\cdot, s; y, t)\}_{j=1}^w|_{[x-R, x] }$}   \Biggr)\\
&\qquad\leq \PP \Bigl( \Big\{\log \cM^+_i(\bm\cdot, s; \bm\cdot, t)|_{\Lambda_i([x, x+K]) \times \Lambda_i([y, y+K]) }\Big\}_{i=1}^n \in \msf A \Bigr), 
\end{split}
\end{equation}
where 
\[
\cM^+_i(\bx, s; \by, t) = \prod_{j=1}^i (x_j-x)^w(y_j-y)^w \cM_i(\bx, s; \by, t),
\]
for any $\bx=(x_1,\ldots, x_i) \in \Lambda_i([x, x+K])$ and $\by=(y_1,\ldots, y_i)\in \Lambda_i([y, y+K])$. Here, $\cM^+$ is defined by removing the entropy factors coming from the factors involving $x$ and $y$ in the Vandermonde in the definition \eqref{e.M definition} of $\cM$.

Observe that the $w=n=1$ case is essentially Theorem~\ref{thm:disj-BK}, except that \eqref{eq:multi-disj-BK} allows the event $\msf A$ to also involve the processes values as the second starting point varies, and the conditioning is additionally done over $\mc Z$ as the starting point varies (apart from the ending point varying already present in Theorem~\ref{thm:disj-BK}). Indeed, these features are already present in the prelimiting Theorem~\ref{thm:disj-BK-log-g-new} and so carry over to the limit in a straightforward way by the same types of arguments as in Section~\ref{s.log gamma to CDRP}.

As mentioned just above, for larger values of $n$ or $w$, a formal proof of \eqref{eq:multi-disj-BK} requires an analog of \Cref{l.K convergence} that gives convergence of multi-path log-gamma partition functions to (appropriately normalized versions of) $\cM_m$ (up to the boundary of $\Lambda_m$) for each $m$, which does not appear to be present in the literature.  In fact, if the starting and ending points are all separate (i.e., one is in the interior of $\Lambda_m$), one can invoke the determinantal formula for $\cK$ (which from \eqref{e.M definition} equals $\cM$ up to a normalization) as in Lemma~\ref{l.K convergence} to get convergence to $\cK$, and if the points are all adjacent, one can use results or arguments as in \cite{corwin2017intermediate} to get convergence to $\cM$. In \eqref{eq:multi-disj-BK}, however, we consider collections of point where some are adjacent and some are separate, and no such general statement is currently available, though we certainly expect it to hold.

\subsection{Higher-indexed curves of the KPZ$_t$ line ensemble}\label{s.generalization.kpz}
Assuming \eqref{eq:multi-disj-BK}, it is plausible that one can also generalize \Cref{thm:disj-BK-line-ensemble-technical}, to bound the higher-indexed curves in the KPZ$_t$ line ensemble conditional on the top several curves.

Namely, take any $w, n\in\N$, $y\in\R$, $K, R\geq 0$. Let $\msf A\subseteq \mc C([y,y+K],\R)^n$ be an increasing Borel measurable set, and for an interval $I$ and $\mb f, \mb g\in \mc C(I,\R)^w$, let $\mb f \geq \mb g$ mean $f_i(x) \geq g_i(x)$ for all $i\in\intint{1,w}$ and $x\in I$. Take $\mb f \in \mc C([y-R,y], \R)^w$, and any $t>0$ and $M > 0$. We expect to have
\begin{equation}   \label{eq:multi-KPZle-BK}
\begin{split}
\MoveEqLeft[12]
\P\left(\Big\{\sum_{j=w+1}^{w+i}\hh_{t,j}|_{[y,y+K]} - Ct^{-1/3}\log M\Big\}_{i=1}^{n} \in \msf A \middd \{\hh_{t,j}|_{[y-R, y]}\}_{j=1}^w = \mb f \right)\\
&\leq \P\left(\Big\{\sum_{j=1}^{i}\hh_{t,j}|_{[y,y+K]}\Big\}_{i=1}^{n} \in \msf A\right) + \frac{3(K+1)t^{2/3}\exp(-cM^2)}{\P( \{\hh_{t,j}|_{[y-R, y]}\}_{j=1}^w \geq \mb f)}.%
\end{split}
\end{equation}
where $C,c>0$ are constants that may depend on $w,n$.
We also expect the same to hold when the conditioning is replaced by $\{\hh_{t,j}|_{[y-R, y]}\}_{j=1}^w \geq f$, in which case we may relax the continuity assumption and allow $\mb f:[y-R,y]\to (\R\cup\{-\infty\})^w$ to be upper semicontinuous. 

Since the sum of the top $m$ lines of the KPZ line ensemble at a given point is the same as $\cM_m$ evaluated when all $m$ starting points coincide and all $m$ ending points coincide, the main input required to generalize \Cref{thm:disj-BK-line-ensemble-technical} to \eqref{eq:multi-KPZle-BK} assuming \eqref{eq:multi-disj-BK} would be a quantitative regularity estimate for $\cM_m$ (in order to relate it with $\cM^+_m$, which is in this context the analog of $\cK_2$ in the $m=2$ case) as the starting and ending points vary, essentially a higher-indexed version of \Cref{l.Z K comparison}. For the $m=2$ case handled in the earlier sections, we were able to make do with the identity \Cref{l.K Z identity} and regularity estimates for the KPZ equation, but such an identity does not seem to generalize to higher indexed curves in a way that is useful for this purpose. We do not pursue obtaining quantitative regularity estimates directly in this paper.

\appendix

\section{Miscellaneous proofs}\label{s.misc proofs}

In this appendix we prove Lemma~\ref{l.conditional monotonicity}, Proposition~\ref{p.lower bound on upper tail}, Lemma~\ref{l.two-point supremum for line ensemble}, and Lemma~\ref{l.two point for kpz sheet}. In Section~\ref{ss:legp} we state some facts about the KPZ line ensemble which will be needed in the proof of Proposition~\ref{p.lower bound on upper tail}, which will be given in Section~\ref{s.lower bound on upper tail}. Section~\ref{s.misc.mono in conditioning} proves Lemma~\ref{l.conditional monotonicity}. The proofs of Lemmas~\ref{l.two-point supremum for line ensemble} and \ref{l.two point for kpz sheet} will be given in Section~\ref{s.regularity proofs}.

\subsection{Line ensembles and Gibbs properties}\label{ss:legp}

W start the following Gibbs properties of $\h_{t,1}$ given $\h_{t,2}$. 
For any $a<b$, denote by $\Fext([a, b])$ the $\sigma$-algebra generated by $\h_{t,1}$ on $\R\setminus (a, b)$, and $\h_{t,2}$.

\begin{lemma} \label{lem:Gibbs}
Take any $a<b$ and $t>0$.
Conditional on $\Fext([a, b])$,
for (1) law of $\fh^\beta_{t,1}$ in $[a,b]$, (2) the rate $2$ Brownian bridge connecting  $\smash{\fh^\beta_{t,1}(a)}$ and $\smash{\fh^\beta_{t,1}(b)}$, the former is absolutely continuous with respect to the latter, with Radon-Nikodym derivative (for a path $B$) proportional to $W(B, \fh^\beta_{t,2})$, where
\begin{equation}\label{e.rn derivative}
W(f, g)=
\exp\Big( - 2\int_a^b \exp(f(x)-g(x))  dx \Big)
\end{equation}
\end{lemma}
This Gibbs property was first introduced and proven for the KPZ line ensemble in \cite{CH14}. The connection between the KPZ$_t$ line ensemble and CDRP was formally established in \cite{nica2021intermediate}.
The form of the Gibbs property presented here is from \cite[Proposition 2.6, Theorem 2.8]{GH22}.

A useful consequence of the Gibbs property is the monotonicity in boundary data recorded below.

\begin{lemma}[Monotonicity in boundary data]\label{l.monotonicity}
Fix $a<b$, real numbers $\smash{w^{*}, z^{*}}\in \R$ and measurable functions $\smash{g^{*}}:[a,b]\to\R \cup \{-\infty\}$ for $*\in \{\uparrow, \downarrow\}$ such that $\smash{w^{\downarrow}\leq w^{\uparrow}}$, $\smash{z^{\downarrow}\leq z^{\uparrow}}$, and, for all $s\in (a,b)$, $\smash{g^{\downarrow}(s)\leq g^{\uparrow}(s)}$. 

For $*\in \{\uparrow, \downarrow\}$, let $\smash{\mathcal{Q}^{*}}$ be a process on $[a,b]$ such that $\smash{\mathcal{Q}^{*}}(a)=\smash{w^{*}}$ and $\smash{\mathcal{Q}^{*}}(b)=\smash{z^{*}}$, with Radon-Nikodym derivative with respect to Brownian bridge given by $W(\mathcal{Q}^{*}, g)$ for $\beta=1$.
Then there exists a coupling of the laws of $\smash{\mathcal{Q}^{\uparrow}}$ and $\smash{\mathcal{Q}^{\downarrow}}$ such that almost surely $\smash{\mathcal{Q}^{\downarrow}}(s)\leq \smash{\mathcal{Q}^{\uparrow}_j(s)}$ for all $s\in (a,b)$.
\end{lemma}

The positive temperature ($\beta=1$) statements are Lemmas~2.6 and 2.7 of \cite{CH14}. See also \cite{dimitrov2022characterization} for a more detailed proof.

The following is a useful correlation inequality saying the line ensemble is positively associated. 

\begin{lemma}[FKG inequality, {\cite[Theorem 2.7]{GH22}}]\label{l.fkg}
For any $t>0$ and $a<b$, and any pair of increasing events $A$ and $B$ in the space of all real continuous functions on $[a, b]$, 
$$\P\left(\h_{t,1}|_{[a,b]} \in A,\, \h_{t,1}|_{[a,b]} \in B\right) \geq \P\left(\h_{t,1}|_{[a,b]} \in A\right)\cdot \P\left(\h_{t,1}|_{[a,b]} \in B\right).$$
\end{lemma}

Finally, we need the following estimate on the one-point lower tail of $\hh_{t,1}$.

\begin{proposition}  \label{lem:fh-lt}\label{l.lower tail}
There exist $c, L_0>0$ such that, for any $0<t\le 1$ and $L>t^{-1/6}L_0$,
\[
\P\left(\hh_{t,1}(0) < -L\right) \leq \exp(-cL^2t^{1/6}).
\]
If we instead assume $t>t_0$ for some $t_0>0$, and $L>L_0$, then
\[
\PP\left(\hh_{t,1}(0)<-L\right)<\exp(-cL^{5/2}),
\]
with the constant $c$ depending on $t_0$.
\end{proposition}

\noindent These two estimates can be deduced from  \cite[Theorem 1.7]{das2023law} and \cite[Theorem 1]{CG20} respectively.

\subsection{Proof of monotonicity in conditioning} \label{s.misc.mono in conditioning}
Here we prove Lemma~\ref{l.conditional monotonicity}.

\begin{proof}[Proof of Lemma~\ref{l.conditional monotonicity}]
\cite[Theorem 2.7]{GH22} states that the law of $\h_{t}$ conditional on $\h_{t,1}(x_j) = y_j^{\shortuparrow}$ stochastically dominates the same conditional on $\h_{t,1}(x_j) = y_j^{\shortdownarrow}$ for any $x_1, \ldots, x_m \in \R$ and $y_1^{*}, \ldots, y_m^{*} \in \R$ for $*\in\{\shortuparrow,\shortdownarrow\}$ such that $\smash{y_j^{\shortuparrow} \geq y_j^\shortdownarrow}$ for all $j\in\intint{1,m}$. Thus the case where $\mc D$ is a finite set is already established. 

Now consider the case where $\mc D$ is an interval. Let $\{\mc D_m\}$ be an increasing sequence of finite subsets of $\mc D$ such that $\cup_{m=1}^\infty \mc D_m$ is a dense subset of $\mc D$. Then for any $k\in\N$ and any event $\msf A\subseteq \mc C(\intint{1,k}\times\mc D, \R)$ of the first $k$ curves of $\h_{t}$ on $\mc D$, it holds that
\begin{equation}\label{e.finite conditional prob}
\P\left(\h_t|_{\intint{1,k}\times\mc D} \in \msf A \midd \h_{t,1}|_{\mc D_m}\right)
\end{equation}
is a martingale in $m$. By Doob's martingale convergence theorem and the continuity of $\h_{t,1}$, this sequence has an almost sure limit, which equals
\begin{align}\label{e.full conditional prob}
\P\left(\h_t|_{\intint{1,k}\times\mc D} \in \msf A \midd \h_{t,1}|_{\mc D}\right).
\end{align}
Let us rephrase this in terms of the probability kernels underlying the conditional probabilities. For each $m\in\N$, let $\mu_m: \mc B\times \R^m \to [0,1]$ be the probability kernel associated to the regular conditional probability in \eqref{e.finite conditional prob}, and $\mu: \mc B\times \mc C(\mc D, \R) \to [0,1]$ be the one associated to \eqref{e.full conditional prob}, where $\mc B$ is the Borel $\sigma$-algebra of $\mc C(\intint{1,k}\times\mc D, \R)$. So, what we have obtained by the martingale argument is that, for each fixed $\msf A\in \mc B$, almost surely,
\begin{align*}
\mu_m(\msf A, \h_{t,1}|_{\mc D_m}) \to \mu(\msf A, \h_{t,1}|_{\mc D}).
\end{align*}
Since $\h_{t,1}$ has full support in $\mc C(\mc D, \R)$ (by the absolute continuity to Brownian motion implied by the Brownian Gibbs property enjoyed by $\h_t$; see Lemma~\ref{lem:Gibbs} and \cite{CH14}), it follows that there exists $\Omega_{\msf A}\subseteq \mc C(\mc D,\R)$ which has probability 1 under the Brownian measure in the statement of the lemma such that, for all $f\in \Omega_{\msf A}$, $\mu_m(\msf A, f|_{\mc D_m}) \to \mu(\msf A, f|_{\mc D}).$

Now let $\msf A$ be an increasing event and let $f,g \in\Omega_{\msf A}$ be such that $f(x)\geq g(x)$ for all $x\in\mc D$. Then by the finite case already established,  $\mu_m(\msf A, f|_{\mc D_m}) \geq \mu_m(\msf A, g|_{\mc D_m})$. Taking $m\to\infty$ yields $\mu(\msf A, f) \geq \mu(\msf A, g)$.

It remains to extend to all increasing events. This is accomplished by an approximation argument (see \cite[Lemma 6]{barbato2005fkg}) from a countable generating set, using the continuity of $\h_{t,1}$ and that $\mc C(\mc D, \R)$ is a Polish space. That is, we guarantee $\mu(\msf A, f) \geq \mu(\msf A, g)$ holds for a suitable countable collection of increasing $\msf A$ for all $f,g\in\Omega$ for a probability one $\Omega\subseteq \mc C(\mc D, \R)$, and this implies the same inequality for all increasing $\msf A$.
\end{proof}

\subsection{Proof of the lower bound on the upper tail}\label{s.lower bound on upper tail}
Here we prove Proposition~\ref{p.lower bound on upper tail}.
\newcommand{\fav}{\msf{Fav}}

\begin{proof}[Proof of Proposition~\ref{p.lower bound on upper tail}]
For an $M$ to be chosen, consider the favourable event 
$$\fav_t = \left\{\hh_{t,1}(-L^{1/2})\geq -L- M\right\}\cap \left\{\hh_{t,1}(L^{1/2})\geq -L - M\right\}.$$
By stationarity and positive association (Lemma~\ref{l.fkg}) of $\hh_{t,1}(x) + x^2$, 
$$\P(\fav_t) \geq \P(\hh_{t,1}(0) > -M)^2.$$
By Proposition~\ref{l.lower tail}, if $M > (t^{-1/6}\vee 1)L_0$, then $\P(\hh_{t,1}(0) > -M) \geq 1-\exp(-cM)$. So by setting $M = C(t^{-1/6}\vee 1)$ for a sufficiently large constant $C$ guarantees that $\P(\fav_t)\geq 1/2$ for all $t>0$.

Consider the $\sigma$-algebra $\F = \Fext(1,[-L^{1/2}, L^{1/2}])$. 
The Brownian Gibbs property says that the distribution of $\hh_{t,1}$ on $[-L^{1/2}, L^{1/2}]$, conditionally on $\F$, is that of a rate two Brownian bridge tilted by the Radon-Nikodym derivative $W_{H_t}/Z_{H_t}$ associated to the conditioned boundary data. By monotonicity (Lemma~\ref{l.monotonicity}), on $\fav_t$, this Brownian bridge stochastically dominates the rate two Brownian bridge $B$ from $(-L^{1/2}, -L-M)$ to $(L^{1/2}, -L-M)$ with no lower boundary condition.

Thus it follows that
\begin{align*}
\P\big(\hh_{t,1}(0) \geq L\big) \geq \E\Big[\P\big(\hh_{t,1}(0) \geq L \mid \F\big)\cdot\one_\fav\Big]  \geq \tfrac{1}{2}\cdot \P\left(B(0) \geq L\right).
\end{align*}
Now $B(0)$ is a normal random variable with mean $-L-M$ and variance $L^{1/2}.$ Thus by a standard lower bound on the normal probability (see, e.g., \cite[Lemma 2.10]{GH22}), it holds on $\fav_t$ that
\begin{align*}
\P\big(\hh_{t,1}(0) \geq L\big) \geq cL^{-3/4} \cdot\exp\left(-\frac{(L+L+M)^2}{2L^{1/2}}\right)%
&\geq \exp(-5L^{3/2}),
\end{align*}
the last inequality for $L>M$. This completes the proof.
\end{proof}

\subsection{Proofs of regularity statements}\label{s.regularity proofs}

Here we give the proofs of Lemmas~\ref{l.two-point supremum for line ensemble} and \ref{l.two point for kpz sheet}. For the proof of Lemma \ref{l.two point for kpz sheet}, we will need a basic two-point estimate. It will be proved along with Lemma~\ref{l.two-point supremum for line ensemble} as both follow immediately from the results of \cite{wu2024applications}.

\begin{lemma}\label{l.two-point for line ensemble}
There exist $C, c>0$ such that, for all $j\in\N$, $x,y\in\R$, and $K>0$,
\begin{align*}
\P\left(|\hh_{t,j}(x) - \hh_{t,j}(y) + x^2-y^2| \geq K|x-y|^{1/2}\right) \leq C\exp(-cK^2).
\end{align*}

\end{lemma}

\begin{proof}[Proofs of Lemmas~\ref{l.two-point for line ensemble} and \ref{l.two-point supremum for line ensemble}]
We start with Lemma~\ref{l.two-point for line ensemble}.
\cite[eq. (2.6)]{wu2024applications} asserts that there exist $c,C>0$ such that, for any $j\in\N$, any line ensemble $\cL$ in a certain class denoted $\mrm{LC}$ (for log-concave), and any $x,y\in\R$,
\begin{align}\label{e.two-point}
\P\left(|\cL_j(x) - \cL_j(y) - (\E[\cL_j(x)] - \E[\cL_j(y)])| \geq K|x-y|^{1/2}\right) \leq C\exp(-cK^2).
\end{align}
The class $\mrm{LC}$ is closed under weak limits as well as constant shifts, and \cite[Proposition 3.4]{wu2024applications} states that the O'Connell-Yor diffusion is a member of $\mrm{LC}$. It thus follows from the convergence of O'Connell-Yor (appropriately centered) to the KPZ line ensemble (see \cite[Theorem 1.2]{nica2021intermediate} or \cite[Proposition 3.7]{CH14} for the result in the case of $j=1$ in closer notation to ours) that the KPZ line ensemble $\h_{t}$ is also a member of $\mrm{LC}$. Since the scaling \eqref{e.scaling for KPZ} to go from $\h_t$ to $\hh_t$ preserves the Wiener measure, $\hh_t$ also lies in $\mrm{LC}$ and thus \eqref{e.two-point} applies with $\L = \hh_t$. Recalling that $\hh_{t,j}(x) + x^2$ is stationary \cite[Proposition 1.4]{amir2011probability} yields the claim.

The proof of Lemma~\ref{l.two-point supremum for line ensemble} is identical after replacing the invocation of \cite[eq. (2.6)]{wu2024applications} with \cite[Theorem 2.8(ii)]{wu2024applications}.
\end{proof}

\begin{proof}[Proof of Lemma~\ref{l.two point for kpz sheet}]
First we prove the two-point tail estimate that there exists $c,C>0$ such that for any $x_1,x_2,y_1,y_2\in \R$ and $M>0$,
\begin{align}\label{e.sheet two point}
\P\left(|\h(x_1,0;y_1,t) - \h(x_2,0;y_2,t)| \geq M\|(x_1,y_1) - (x_2,y_2)\|^{1/2}\right) \leq C\exp(-cM^2).
\end{align}
From \eqref{e.sheet two point} we obtain the statement of Lemma~\ref{l.two point for kpz sheet} by a standard chaining argument (e.g., \cite[Lemma 3.3]{DV21}).
To prove \eqref{e.sheet two point}, we rewrite the inequality from Lemma~\ref{l.two-point for line ensemble} as
\begin{align}\label{e.to prove for sheet two-point}
\P\left(|\h(x,0;y_1,1) - \h(x,0;y_2,1)| \geq M|y_1 - y_2|^{1/2}\right) \leq C\exp(-cM^2).
\end{align}
Then, by (i) the triangle inequality, (ii) since $|a|+|b|\leq 2^{p-1}(a^4+b^4)^{1/4}$ for any $a,b\in\R$, and (iii) since  $x\mapsto \h(x,0;y,1)$ is equal in distribution to $x\mapsto \h(y,0;x,1)$,
\begin{align*}
|\h(x_1,0;y_1,1) - \h(x_2,0;y_2,1)|
&\leq |\h(x_1,0;y_1,1) - \h(x_1,0;y_2,1)| + |\h(x_1,0;y_2,1) - \h(x_2,0;y_2,1)|
\\
&\stackrel{d}{=} |\h(x_1,0;y_1,1) - \h(x_1,0;y_2,1)| + |\h(y_2,0;x_1,1) - \h(y_2,0;x_2,1)|\\
&\leq M\left(|y_1-y_2|^{1/2} + |x_1-x_2|^{1/2}\right) \leq \tilde CK\|(x_1,y_1) - (x_2,y_2)\|^{1/2},
\end{align*}
where the penultimate inequality holds on an event of probability at least $1-C\exp(-cM^2)$ by applying \eqref{e.to prove for sheet two-point}. This proves \eqref{e.sheet two point} and completes the proof of Lemma~\ref{l.two point for kpz sheet}. 
\end{proof}

\bibliographystyle{alpha}
\bibliography{bibliography}

\end{document}